\documentclass{amsart}

\usepackage{amsmath}
\usepackage{amssymb}
\usepackage{graphicx}
\usepackage{subfigure}

\usepackage{amsthm}
\usepackage{physics}

\newcommand{\V}[1]{\ensuremath{\mathbf{#1}}}

\newcommand{\iid}{\stackrel{\text{iid}}{\sim}}

\newcommand{\T}{\ensuremath{{\scriptscriptstyle \top}}}
\newcommand{\R}{\ensuremath{\mathbb{R}}}

\newcommand{\Nc}{\ensuremath{\mathcal{N}}}

\newcommand{\MD}{MD}

\theoremstyle{theorem} \newtheorem{thm}{Theorem}
\theoremstyle{theorem} \newtheorem{cor}{Corollary}

\theoremstyle{definition} \newtheorem{defn}{Definition}
\theoremstyle{definition} \newtheorem{assum}{Assumption}

\begin{document}

\title[OS of precision matrix for MD in manifold learning]{Optimal Recovery of {Precision Matrix for} Mahalanobis Distance {from High Dimensional Noisy Observations} in Manifold Learning}

\author{Matan Gavish}
\address{The School of Computer Science and Engineering, Hebrew University of Jerusalem}
\email{gavish@cs.huji.ac.il}
\author{Ronen Talmon}
\address{The Department of Electrical Engineering, Technion -- Israel Institute of Technology}
\email{ronen@ef.technion.ac.il}
\author{Pei-Chun Su}
\address{The Department of Mathematics, Duke University}
\email{b94401079@gmail.com}
\author{Hau-Tieng Wu}
\address{The Department of Mathematics and Department of Statistical Science, Duke University}
\email{hauwu@math.duke.edu}

\maketitle

\begin{abstract}
{Motivated by establishing theoretical foundations for various manifold learning algorithms, we study the problem of Mahalanobis distance (MD){, and the associated precision matrix,} estimation from high-dimensional noisy data.
  By relying on recent transformative results in covariance matrix estimation, we demonstrate the sensitivity of \MD~{and the associated precision matrix} to measurement noise, determining the exact asymptotic signal-to-noise ratio at which \MD~fails, and quantifying its performance otherwise.
  In addition, for an appropriate loss function, we propose an asymptotically optimal shrinker, which is shown to be beneficial over the classical implementation of the \MD, both analytically and in simulations. {The result is extended to the manifold setup, where the nonlinear interaction between curvature and high-dimensional noise is taken care of. The developed solution is applied to study a multiscale reduction problem in the dynamical system analysis.}}
Mahalanobis distance, large $p$ large $n$, optimal shrinkage, {precision matrix}\\
2000 Math Subject Classification: 34K30, 35K57, 35Q80,  92D25
\end{abstract}

\section{INTRODUCTION}\label{sec:introduction}

High-dimensional datasets encountered in modern science often
exhibit nonlinear low-dimensional structures. 
One prominent approach to deal with such point clouds is to model their nonlinear structures by manifolds.
In the last two decades, this direction has led to the emergence of a multitude of {\em manifold learning} methods, including the classical ISOMAP \cite{tenenbaum2000global}, locally linear embedding (LLE) \cite{Roweis_Saul:2000}, Hessian LLE \cite{Donoho_Grimes:2003}, eigenmap \cite{belkin2003laplacian}, and diffusion maps \cite{coifman2006diffusion}, as well as the more recent vector diffusion maps \cite{Singer_Wu:2012}, multiview diffusion maps \cite{Linderbaum2015}, and alternating diffusion maps \cite{lederman2015learning,talmon2016latent}, to name but a few.
Typically in manifold learning, point clouds in $\mathbb{R}^p$ are assumed to be sampled from a $d$-dimensional smooth manifold $\mathcal{M}$ embedded in $\mathbb{R}^p$, usually with some additional contaminating noise. In this setting, the manifold represents the ``essence'' or the ``signal'' of the data.
Consequently, the goal in manifold learning is to recover the geometric or topological structure of the manifold from the data points, and in turn, to use the recovered structure to embed the high-dimensional data in a low-dimension space, facilitating a compact and informative representation of the data.
This approach has been successfully applied to applications from a broad range of fields, e.g. dynamical systems modeling \cite{talmon2015manifold,yair2017reconstruction}, sleep stage assessment \cite{wu2015assess,liu2018diffuse}, cryo electro microscope \cite{singer2018mathematics}, image denoising \cite{singer2009NLM}, single channel blind source separation like fetal electrocardiogram analysis \cite{su2019recovery} and stimulation artifact removal for the intracranial electroencephalogram \cite{alagapan2018diffusion}, and long term physiological signal visualization and analysis \cite{lin2021wave,wang2020novel}.

Classical manifold learning methods heavily rely on meaningful measures of 
pairwise discrepancy between data points.   
In this so-called {\em metric design} problem, the data analyst
aims to find a useful {\em metric} representing the relationship between data points
embedded in a high-dimensional space. 
In this paper, we study the Mahalanobis Distance (MD) -- a popular, and arguably the first
method for metric design \cite{mahalanobis1936generalized, McLachlan1999}.
 \MD~ was originally proposed in 1936 
with the classical low-dimensional setting in mind, namely, for the case where
the ambient dimension $p$ is much larger than the dataset size $n$. 
Interestingly, due to its useful statistical and invariance properties, \MD~ 
became the basis of several geometric data analysis techniques \cite{Yang2006DistanceML,
Weinberger:2009:DML:1577069.1577078, Xiang2008LearningAM}, aimed specifically at the high-dimensional regime $p\asymp n$.

In a recent line of work \cite{singer2008non,talmon2013empirical}, a variant of \MD~was proposed and used to reveal hidden intrinsic manifolds underlying high-dimensional, possibly multimodal, observations. The main purpose of \MD~in this hidden manifold setup is handling possible deformations caused by the observation or sampling process. Broadly, this is carried out by estimating a quadratic form of the Jacobian of the (unknown) observation function, which is equivalent to estimating the precision matrix \emph{locally} on the manifold. 
It was recently shown that \MD~is also implicitly used in the seminal LLE 
algorithm \cite{Roweis_Saul:2000}, when the barycenter step is properly expressed \cite{2018arXiv180402811M}. 

As the number of dimensions $p$ in typical data analysis applications continues to
grow, it becomes increasingly crucial to understand the behaviour of \MD, as well
as other metric design algorithms, in the high-dimensional regime $p\asymp n$.
At first glance, it might seem that this regime poses little more than a
computational inconvenience for metric design using \MD. Indeed, it is
easy to show that in the absence of measurement noise, \MD~cares little
about the increase in the ambient dimension $p$.  
This paper however calls attention to the following key observation. In the high-dimensional regime $n\asymp p$, 
{\em in the presence of ambient measurement noise}, a new phenomenon emerges, which introduces {various} nontrivial effects on the performance of \MD. 
Depending on the noise level, in the high-dimensional regime, \MD~may be adversarially affected or even fail completely. 
Clearly, the assumption of measurement noise cannot be realistically excluded, and yet, to the best of our knowledge, this phenomenon has not been 
 previously fully studied. A first step in this direction was taken in 
 \cite{10.1007/978-3-319-73241-1_4}, 
 with the calculation of the distribution of \MD~under specific assumptions.

Let us describe this key phenomenon informally at first.
The computation of \MD~involves estimation of the inverse covariance matrix, or precision matrix, corresponding to the
data at hand. Classically, the estimation relies on the sample covariance, which is inverted using the Moore-Penrose pseudo-inverse.
It is well-known that, in the {high dimensional setup, or in the} regime $n\asymp p$, the sample covariance matrix is a poor estimator of the
underlying covariance matrix. 
Indeed, advances in random matrix theory from the last decade imply that the eigenvalues and
eigenvectors of the sample covariance matrix are both {biased}, namely, do not
converge to the corresponding eigenvalues and eigenvectors of the underlying population
covariance matrix \cite{Paul2007,johnstone2001}. 
Such {biases} in small eigenvalues, which lead to inaccurate covariance matrix estimation, become immense when applying the Moore-Penrose pseudo-inverse.

{This challenge in the high dimensional setup is amplified when we have a nonlinear manifold structure. Inverting small (inconsistent) eigenvalues is challenging when evaluating the precision matrix, and this issue becomes more challenging}
in the context of manifold learning, where the estimation of MD is performed locally \cite{singer2008non,talmon2013empirical}. Ideally, under the manifold assumption, in infinitely small neighborhoods and for a sufficiently large number of samples {without noise}, the rank of the local sample covariance equals the dimension of the manifold, which is typically smaller than the dimension of the ambient space, making the sample covariance low-rank with distinct strictly zero eigenvalues. However, in practice, due to the finite sample set, the considered neighborhoods cannot be sufficiently small. In such cases, depending on the manifold curvature, samples from the manifold depart from the tangent space to the manifold, and the rank of their sample covariance matrix undesirably grow, {where the eigenvalues related to curvature are much smaller compared with those related to the tangent space. When noise exists, the situation is further complicated by the interaction between those curvature related small eigenvalues and the relatively few points compared with the ambient space dimension.}

In this paper, we study this problem and propose a remedy. By relying on formal existing results in covariance matrix estimation, we measure the sensitivity of \MD~to measurement noise. 
Under the assumption that locally the data on the manifold lie on a low-dimensional
linear subspace embedded in the ambient space $\R^p$ and that the measurement
noise is Gaussian white, we are able to determine the exact asymptotic signal-to-noise ratio at which \MD~fails, and quantify its performance otherwise.
{Then, we propose a better \MD~estimator based on the idea of {\em shrinkage} of the associated precision matrix.} It has been known since the 1970's  
\cite{Stein1986} that by shrinking the sample covariance eigenvalues one can
significantly mitigate the noise effects and improve the covariance estimation
in high-dimensions. We formulate the classical \MD~as a particular choice of 
shrinkage estimator for the eigenvalues of the sample covariance.
Building on recent results in high-dimensional covariance estimation, including the general theory in \cite{2013arXiv1311.0851D} and a special case with application in random tomography \cite[Section 4.4]{SingerWu2012}, we find an {\em asymptotically optimal shrinker} for the precision matrix estimation, which is better than the classical implementation of the \MD, whenever \MD~is computed from noisy high-dimensional data.
We show that under a suitable choice of a loss function for the estimation of \MD, our
shrinker is the unique asymptotically optimal shrinker; the improvement in
asymptotic loss it offers over the classical \MD~is calculated exactly. {We then extend the above established results to handle the challenge of designing a better metric when applying diffusion map. This extension is nontrivial due to the nontrivial interaction of curvature and noise, and the finite sampling size. Finally, we apply diffusion maps with the proposed estimate of \MD to separate slow and fast dynamics when the observation is contaminated by high dimensional noise.}

While the present paper focuses on \MD, we posit that the same phenomenon holds
much more broadly and in fact affects several widely-used manifold learning, particular in metric learning algorithms. In this regard, the present paper seeks to highlight the fact that manifold learning and metric learning algorithms will
{\em not} perform as predicted by the noiseless theory in high dimensions, and may fail completely beyond a certain noise level. 

\section{PROBLEM SETUP} \label{sec:setup}


\subsection{Manifold model}\label{Section:ManifoldModel}

When a point cloud $\mathcal{X}:=\{x_i\}_{i=1}^n\subset \mathbb{R}^p$ has a nontrivial nonlinear structure, or even nontrivial topological structure, a common approach is to model this structure by a manifold.
This is known as the so-called \emph{manifold assumption}. {Such manifold assumption holds for various practical data. Examples include cryo-electro microscopy \cite{singer2018mathematics}, phase spaces of dynamical systems \cite{talmon2015manifold,yair2017reconstruction}, and various biomedical signals \cite{su2019recovery,alagapan2018diffusion,lin2021wave,wang2020novel}. The main feature of this manifold assumption is that the points are distributed on a nonlinear set so that they are nonlinearly related, which generalizes the commonly used linear model.}

{To model $\mathcal{X}$ by the manifold model, consider a $p$-dimensional random vector $X:\Omega \rightarrow \mathbb{R}^p$, which is a measurable function with respect to the probability space $(\Omega,\mathcal{F},\mathbb{P})$, where $\mathbb{P}$ is the probability measure defined on the sigma algebra $\mathcal{F}$ in {the event space} $\Omega$. We assume} that the range of $X$ is supported on a $d$-dimensional compact, smooth Riemannian manifold $(M,g)$ isometrically embedded in $\mathbb{R}^p$ via $\iota:M\hookrightarrow \mathbb{R}^p$. 
In this work, we assume that $M$ is boundary-free to simplify the exposition. 
{We shall mention that the commonly considered linear model is a special manifold model, where the manifold is an affine linear subspace space of $\mathbb{R}^p$.}


On the manifold, the associated statistical setup is as follows. Let $\tilde{\mathcal{B}}$ be the Borel sigma algebra of $\iota(M)$, and let $\tilde{\mathbb{P}}_X$ denote the probability measure induced from $X$. Clearly, $\tilde{\mathbb{P}}_X$ is defined on $\tilde{\mathcal{B}}$. {Denote $d V$ to be the Riemannian volume density of $M$ associated with the metric $g$.}
For simplicity, we assume that $\tilde{\mathbb{P}}_{X}$ is absolutely continuous with respect to the induced Riemannian measure on $\iota(M)$, denoted by $\iota_* dV$. 
By the Radon-Nikodym theorem, for any $z \in \iota(M) \subset \mathbb{R}^p$, there exists a non-negative measurable function $P(z)$ defined on $\iota(M)$ such that $d \tilde{\mathbb{P}}_{X}(z)=P(z)\iota_* d V(z)$.
The probability density function (pdf) of $X$ on $M$ is defined to be $P(z)$. We further assume that $P(z)$ is bounded away from $0$ and smooth. When $P(z)$ is constant, we call $X$ a \textit{uniform} random sampling scheme; otherwise it is \textit{nonuniform}. {Now we introduce the key quantity of interest in this paper, the local covariance matrix.}

\begin{defn}
Fix $x\in M$. For an open simply connected neighborhood of $x$, $\mathcal{O}(x)\subset\mathbb{R}^p$, define
\begin{equation}\label{eq:local_covariance}
\Sigma_{\mathcal{O}(x)}:=\mathbb{E}[(X-\mu_{\mathcal{O}(x)})(X-\mu_{\mathcal{O}(x)})^{\top}\chi_{\mathcal{O}(x)}(X)]\in\mathbb{R}^{p\times p}
\end{equation}
as the \textit{local covariance matrix} associated with $\mathcal{O}(x)$ centered at $\mu_{\mathcal{O}(x)}$, where
\[
	\mu_{\mathcal{O}(x)} := \mathbb{E}[X\chi_{\mathcal{O}(x)}(X)] = \frac{1}{|\mathcal{O}(x)\cap \iota(M)|}\int_{\mathcal{O}(x)\cap \iota(M)}zP(z)dz\,,
\]
$|\mathcal{O}(x)\cap \iota(M)|$ is the volume of $\mathcal{O}(x)\cap \iota(M)$, and $\chi_{\mathcal{O}(x)}$ is the indicator function of the set $\mathcal{O}(x)$.
\end{defn}

{One main goal of considering this local covariance is capturing those directions with maximal variation of the dataset. From the knowledge of principal component analysis, those directions are the eigenvectors of the local covariance matrix with the largest eigenvalues.} We have some remarks. {First, by Nash's isometric embedding theorem, $\Sigma_{\mathcal{O}(x)}$ is of rank $D$, where $D\leq d(3d+11)/2$ for any $\mathcal{O}(x)$. Second,} in existing literature, there is another different definition of the local covariance matrix, in which the mean $\mu_{\mathcal{O}(x)}$ is replaced with $\iota(x)$. In general, the two definitions are different, even when the set is perfectly symmetric, e.g. $\mathcal{O}(x)$ is a ball, and $P(x)$ is uniform. Indeed, since the range of $X$ is supported on $\iota(M)$, when $\iota(M)$ is not flat, $\iota(x)$ might deviate from the center of $\mathcal{O}(x)$ due to the curvature of the manifold; that is, $\frac{1}{|\mathcal{O}(x)\cap \iota(M)|}\int_{\mathcal{O}(x)\cap \iota{M}}(z-\iota(x))dz \neq 0$. While this seems to be a problem, it was shown in \cite{2018arXiv180402811M} that the difference between the center of $\mathcal{O}(x)$ and $\iota(x)$ is negligible (expressed as a higher order term in the error) when the diameter of $\mathcal{O}(x)$ is sufficiently small. 
Broadly, since locally the manifold can be well approximated by an affine subspace, when the diameter of $\mathcal{O}(x)$ is sufficiently small, the data located in $\mathcal{O}(x)$ can be well approximated by the tangent space to the manifold at $x$, i.e., $T_xM$. This point will be further addressed in Section \ref{sec:linearSpike}.

The above derivation leads to the following local statistical model, which enables to further study the local structure of a manifold. For $x\in M$ and an open, simply connected neighborhood of $x$, $\mathcal{O}(x)\subset \mathbb{R}^p$, we define a new random vector
\begin{equation}\label{ManifoldModel:X}
X_x:=X\chi_{\mathcal{O}(x)}{(X)}, 
\end{equation}
with mean $\mu_x := \mu_{\mathcal{O}(x)}$ and covariance matrix $\Sigma_x := \Sigma_{\mathcal{O}(x)}$. By definition, $X_x$ is a bounded random vector, and hence, all its moments are finite. Also, as has been widely discussed in the literature (see \cite{2018arXiv180402811M} and reference therein), the first $d$ dominant eigenvectors of $\Sigma_x$ form an accurate estimate of the tangent space to the manifold at $x$. Specifically, if $\Sigma_xu_l=\lambda_lu_l$, for $l=1,\ldots,p$, where $\lambda_1\geq \lambda_2\ldots\geq\lambda_p$, then the span of $\{u_1,\ldots,u_d\}$ approximates $\iota_*T_xM$.

Often in applications, the manifold is not directly accessible due to additional noise, and we can only sample
\begin{equation}
Y=X+\sigma \xi,
\end{equation}
where $\xi\sim \mathcal{N}(0,I_p)$. As a consequence, for $x\in M$, in the presence of noise, the corresponding local random vector can be recast as
\begin{equation}\label{eq:manifold_y}
Y_x:=Y\chi_{\mathcal{O}(x)}{(X)},
\end{equation}
with mean $\mu_x := \mu_{\mathcal{O}(x)}$ and covariance matrix $\Sigma_x +\sigma^2 I_p$. Although $Y_x$ is not a bounded random vector, all its moments are finite.

We remark that the local set $\mathcal{O}(x)$ can be defined in several plausible ways, depending on the problem at hand. A common choice is the ball $B_x(\epsilon)$ with the center at $\iota(x)$ and the radius $\epsilon>0$. In other applications, $\mathcal{O}(x)$ might be an ellipsoid \cite{singer2008non,talmon2013empirical,2018arXiv180402811M} or a more general setup, depending on the metric of interest. We will revisit to this issue in the sequel.

\subsection{Linear spiked model}\label{sec:linearSpike}

In manifold learning, local kernels are commonly used, e.g., kernels based on radial basis functions. The use of kernels implies that only points around the center of the kernel $x\in M$ contribute to the algorithm outcome. Therefore, considering those points located inside the neighborhood of $x$, $\mathcal{O}(x)$, is sufficient for analyzing a manifold learning algorithm with a local kernel. Since the bandwidth of the kernel is typically small, the diameter of the local set $\mathcal{O}(x)$ is small as well. Consequently, by the definition of a manifold, data in $\mathcal{O}(x)$ can be well approximated by the tangent space to the manifold at $x$, {which is a low dimensional affine subspace}.
{Note that in the special case that the manifold is a linear affine subspace, all points are located by a low dimensional space.} 

Since {\em locally} the manifold can be viewed merely as a linear subspace, the local statistical model described in Section \ref{Section:ManifoldModel} can be well approximated by the classical linear spiked model {(or spiked covariance model \cite{johnstone2001})}, which is detailed next with slight modifications in the notation.  
{We note that in this section, deviations of the samples from a linear space due to the manifold curvature will not be treated, and their affect will be considered together with the affect of the ambient noise. However, we will extend and test the ability of the proposed estimator to handle such phenomena in the simulation study in Section \ref{sec:simulation}.}

{We now consider the spike model. In plain English, a spike model is a manifold model when the manifold is an affine subspace, where the dimension of the affine subspace is fixed.} Consider a point cloud in $\R^p$ supported on a $d$-dimensional linear subspace, where $d\leq p$. 
For simplicity, we assume that the point cloud is sampled independently and identically (i.i.d.) from 
\begin{align}\label{SpikeModel:X}
	X \sim \Nc \left( \V{\mu}, \Sigma_X \right) \in \mathbb{R}^p\,,
\end{align}
where $\V{\mu}$ denotes the mean and $\Sigma_X$ denotes the population covariance matrix, whose rank is equal to $d$. Note that we could consider a random vector with mean $0$ and a finite fourth moment \cite{2013arXiv1311.0851D}; since this could only increase the notational burden without providing additional insights,  we focus on this simplified model.

It is often convenient to note that a point cloud sampled from $X$ can be understood as sampling i.i.d. from a $p$-dim random vector 
\begin{equation}\label{Equation:rewrittenX}
X=\V{\mu}+\sum_{l=1}^d\sqrt{\lambda_l}\zeta_l u_l, 
\end{equation}
where $\lambda_l>0$, $\zeta_l\sim \mathcal{N}(0,1)$, $\mathbb{E}\zeta_l\zeta_k=\delta_{l,k}$, $\Sigma_Xu_l=\lambda_lu_l$ and $\|u_l\|_{L^2}=1$ for $l=1,\ldots,d$. Thus, the $d$-dimensional linear affine subspace is the space spanned by $u_1,\ldots,u_d$, {which could be understood as spikes and hence the nomination of the model,} and shifted by $\V{\mu}$. We note that this global linear model is related to the local structure of a manifold in \eqref{ManifoldModel:X} in the following way: $\mu$ here is the center point $x$ on the manifold in \eqref{ManifoldModel:X}, and the $d$-dimensional linear affine subspace spanned by $u_1,\ldots,u_d$ is the tangent space at $x${, while we need extra components to capture the curvature of the manifold}.

Similarly to the manifold model, suppose that the samples of $X$, which we refer to as the {\em signal}, are not directly
observable. Instead, the observed data consist of samples from the random variable
\begin{equation}\label{SpikeModel:Y}
    Y = X + \sigma \xi\,,
\end{equation}
where $0\leq \sigma<\infty$ and $\xi\sim\Nc(0,I_p)$ is a Gaussian measurement noise independent of $X$, which we assume for simplicity to be white.

\section{{PRECISION MATRIX AND} MAHALANOBIS DISTANCE ESTIMATION}
\label{sec:mahalanobis_estimation}

The focal point of this paper is the estimation of the \MD~under the manifold model from a statistical standpoint, which, as described in Section \ref{sec:setup}, leads to the classical linear spiked model.
The estimation of the \MD~involves the estimation of the local precision matrix. Therefore, we start this section with details on the estimation of the precision matrix in Section \ref{sec:precision}, followed by a detailed description of the estimation of the \MD~in Section \ref{sec:Mahalanobis_dist}.
 
\subsection{Precision matrix estimation {and its challenge}}
\label{sec:precision}

In Section \ref{Section:ManifoldModel} and Section \ref{sec:linearSpike}, we showed that the local and global statistical models are seemingly very similar.
Indeed, at first glance, both models consist of a hidden ``signal'' component (\eqref{ManifoldModel:X} and \eqref{SpikeModel:X}) and noisy accessible observations (\eqref{eq:manifold_y} and \eqref{SpikeModel:Y}).
Furthermore, under both models, the local and global population covariance matrices of the ``signal'' component ($\Sigma_x$ and $\Sigma_X$) are approximately low rank and precisely low rank, respectively. 

However, one of the main claims of this work is that although the two models seem equivalent, subtle differences between them, particularly in the context of precision matrix estimation, become fundamental.

Let us consider first the estimation of the precision matrix in the simpler, classical linear spiked model described in \eqref{SpikeModel:X} and \eqref{SpikeModel:Y}.
Without noise, the computation of the precision matrix of $X$ could be simply implemented using the Moore-Penrose pseudo-inverse. Suppose the eigendecomposition of $\Sigma_X$ is given by $\Sigma_X=U\textup{diag}(\lambda_1,\ldots,\lambda_d,0,\ldots,0)U^\top$ where $\lambda_1\geq\lambda_2\geq\ldots$ and $U\in O(p)$. Then, the pseudo-inverse is given by $\Sigma_X^{\dagger}=U\textup{diag}(1/\lambda_1,\ldots,1/\lambda_d,0,\ldots,0)U^\top$, namely, inverting the non-zero eigenvalues.
The introduction of additive noise poses a significant challenge for such an estimation, since the small eigenvalues could be mixed with the noise. While the contribution of such small eigenvalues is limited in the composition of the covariance matrix, in the composition of the precision matrix their affect becomes significant.

In addition to this classical challenge, the estimation of the precision matrix under the manifold model poses another layer of complexity. 
To simplify the discussion, we assume that $\mathcal{O}(x)$ is a simply connected ball, that is $\mathcal{O}(x)=B_{\epsilon}^{\mathbb{R}^p}(\mu_x)$, where $B_{\epsilon}^{\mathbb{R}^p}(\mu_x)$ is a Euclidean ball in $\mathbb{R}^p$ of radius $\epsilon>0$ centered at $\mu_x$ with a sufficiently small $\epsilon$. In this case, the geometric picture of the local covariance matrix is well captured by \cite[Proposition 3.1]{Wu_Wu:2017}, which is summarized as follows. Fix $x\in M$. Assume that the manifold is translated and rotated properly, so that $x=0$ and the tangent space in $\mathbb{R}^p$, $\iota_*T_x M$, is spanned by $\{e_1,\ldots,e_d\}$, where $e_j$ is a unit $p$-dim vector with $1$ in the $j$-th entry. We have the following asymptotical expansion of \eqref{ManifoldModel:X} when $\epsilon>0$ is sufficiently small:
\begin{align}\label{Expansion:LocalCov}
\Sigma_{B_{\epsilon}^{\mathbb{R}^p}(\mu_x)} =\frac{|S^{d-1}|  {P}(x)}{d(d+2)} \epsilon^{d+2}
\Big(\begin{bmatrix}
I_{d \times d} & 0 \\
0& 0  \\
\end{bmatrix}+
O(\epsilon^{2})\Big) \,,
\end{align}
where $S^{d-1}$ is a $(d-1)$-sphere, {$|S^{d-1}|$ is the volume of $S^{d-1}$,} and the implied constant in $O(\epsilon^2)$ depends on the second fundamental form at $x$.
From a statistical perspective, this expansion well captures the intuition that the variability of the data located in $B_{\epsilon}^{\mathbb{R}^p}(\mu_x)$ is restricted to only few directions aligned with $T_xM$. Indeed, by applying the perturbation theory to \eqref{Expansion:LocalCov}, the $d$ eigenvectors of $\Sigma_{B_{\epsilon}^{\mathbb{R}^p}(\mu_x)}$ corresponding to the largest $d$ eigenvalues provide an approximate basis for the embedded tangent plane $\iota_*T_xM$. 
The expression of the covariance matrix in \eqref{Expansion:LocalCov} implies that there is a significant spectral gap between the top $d$ eigenvalues and the remaining ones, which depend on the curvature of the manifold. Furthermore, these remaining eigenvalues could be very small with order higher than $\epsilon^{d+4}$. 
In other words, under the manifold setting, in contrast to the linear spiked model, even in the noiseless case, the covariance is not strictly low rank with only $d$ nonzero eigenvalues. Such small eigenvalues due to the curvature of the manifold, {if not properly taken care before taking inverse to estimate the precision matrix, the obtained precision matrix might be deformed. Such small eigenvalues can be} added to the noise components and together they obscure the spectral gap. {Thus, it} is necessary for determining the ``effective rank'' of the matrix{, which is} essential for the calculation of the precision matrix. 
For more details, we refer the readers to the detailed discussion in \cite{Wu_Wu:2017,2018arXiv180402811M}.

In addition to the above challenge due to the curvature, the presence of noise particularly in the manifold setting imposes another challenge -- that is, how to find the {\em true} neighbors? Specifically, note that the {\em neighboring points} of $x$ denoted by $\chi_{\mathcal{O}(x)}$ need to be identified from the noisy samples.
When $\sigma=0$, i.e. in the noiseless case, a neighbor can be easily identified if the diameter of $\mathcal{O}(x)$ is sufficiently small. However, when $\sigma>0$, i.e., in the presence of noise, it is not clear if a neighbor determined from the noisy point cloud is truly a neighbor. Concretely, let $\mathcal{X}=\{x_i\}_{i=1}^n\subset \iota(M)\subset \mathbb{R}^p$ denote a set of identical and independent (i.i.d.) random samples from $X$, and $\mathcal{Y}=\{y_i\}$, where $y_i=x_i+\sigma \xi_i$ is sampled from $Y$. Then, in general $\|y_i-y_j\|_{\mathbb{R}^p}\leq \epsilon$ does not imply $\|x_i-x_j\|_{\mathbb{R}^p}\leq \epsilon$. Thus, a naive approach to determine neighbors might fail. To the best of our knowledge, there are only few existing algorithms for determining neighbors when the point cloud is noisy. For example, determining a neighbor by the diffusion distance \cite{coifman2006diffusion}, which has a solid theoretical support {when noise exists} \cite{ElKaroui:2010a,ElKaroui_Wu:2016b}. Since finding neighbors is still a challenging problem on its own,
in this paper, we focus only on estimating the precision matrix, and subsequently, the \MD, assuming that the {\em true} neighbors are known.

\subsection{Mahalanobis distance}
\label{sec:Mahalanobis_dist}

We are now ready to define the \MD: we start with the definition of the  \MD~under the manifold model, and then for reference, we present the typical definition under the global linear spiked model.

Under the manifold model, since the covariance matrix may have no strictly zero eigenvalues even in the noiseless case due to the curvature, we consider the following definition of the \MD.

\begin{defn}[Mahalanobis distance under the manifold model]\label{Definition:MD:manifoldModel}
Suppose $x\in M$. The \MD~between $z \in \iota(M)$ and $X_x$ is defined by
\begin{equation} \label{trueManifold:eq}
    d_{\Sigma_x}^2(z, X_x) :=  
	(z - \mu_x)^\top \mathcal{I}_d( \Sigma_x) (z - \mu_x),
\end{equation}
where $\mathcal{I}_d( \Sigma_x)$ is the {\em truncated pseudo-inverse of degree $d$} defined by 
\[
\mathcal{I}_d( \Sigma_x):=U\textup{diag}(1/\lambda_1,\ldots,1/\lambda_d,0,\ldots,0)U^\top, 
\]
$\Sigma_x=U\textup{diag}(\lambda_1,\ldots,\lambda_p)U^\top$ is the eigendecomposition of $\Sigma_x$, and $\lambda_1\geq\lambda_2\geq\ldots$.
\end{defn}
Note that the knowledge of the manifold dimension $d$ is required for this definition; yet, it is in general not available and needs to be estimated. 

For comparison, consider also the classical definition of the \MD~under the global linear spiked model.
\begin{defn}[Mahalanobis distance under the linear spiked model]\label{Definition:MD under linear spike}
The \MD~between an arbitrary point
${z}\in\mathbb{R}^p$ and the underlying signal distribution $X$ \eqref{SpikeModel:X} is defined by
\begin{equation} \label{true:eq}
    d_{\Sigma_X}^2({z}, X) =  
	({z} - {\mu})^\top  \Sigma_X^{\dagger} ({z} - {\mu}),
\end{equation}
where $\dagger$ denotes the Moore-Penrose pseudo-inverse.
\end{defn}

Let us take a closer look at the latter definition. Since $\Sigma_{X}$ is semi-positive definite, by the Cholesky decomposition, we have $\Sigma_{X}^\dagger=WW^\top$, where
$W\in \mathbb{R}^{p\times d}$.
Hence 
\begin{align}
    d^2_{\Sigma_X}({z}, X)= ({z} - {\mu})^\top (WW^\top )({z}-{\mu})
	= \|W^\top ({z}- {\mu})\|_{\mathbb{R}^p}^2\,, \nonumber
\end{align}
which indicates that geometrically, \MD~evaluates the relationship between ${z}$ and (the mean of) $X$ by a proper linear transform. In Section \ref{sec:eig}, we relate $W^\top$ to the inverse of the Jacobian of arbitrary unknown observation functions and show that it gives rise to an important invariance property of the \MD~in the context of manifold learning.
Here, we only demonstrate a primary merit of \MD, which stems from its invariance to rotation and rescaling. Importantly, this invariance property holds for both definitions: under the linear spiked model as well as under the manifold model.
Consider the linear spiked model and a random variable $\widetilde{X} =
cAX$, where $c\in \mathbb{R}$ models rescaling and $A \in O(p)$ models rotation. Here $O(p)$ denotes the
group of $p$-by-$p$ orthogonal matrices. The population covariance matrix of $\widetilde{X}$ is
\begin{equation}\label{Sigma_Z:eq}
\Sigma_{\widetilde{X}} = c^2A \Sigma_X A^\top \,
\end{equation} 
and its population mean is rotated and rescaled to ${\tilde{\mu}}=cA{\mu}$. 
To demonstrating the invariance, suppose ${\tilde{z}}=cA{z}$ and observe the \MD~between ${\tilde{z}}$ and $\widetilde{X}$:
\begin{align}
    d_{\Sigma_{\widetilde{X}}}^2({\tilde{z}}, \widetilde{X})
    \nonumber & = (cA({z} - {\mu}))^\top  \Sigma_{\widetilde{X}}^{\dagger} (cA({z} - {\mu})) = ({z} - {\mu})^\top  \Sigma_X^{\dagger} ({z} - {\mu})=d_{\Sigma_X}^2({z}, X).
\end{align}
The same argument holds for the \MD~under the manifold model; for brevity, we omit the details.

Now, recall that the goal in this work is to estimate the \MD~between a point $z\in\mathbb{R}^p$ and $X$, when we only have access to noisy data sampled from $Y$ in \eqref{SpikeModel:Y}.
Concretely, assume that ${y}_1,\ldots,{y}_n \iid Y$ is a sample of $n$ data points. 
Since $\Sigma_X$ is unknown, the quantity $d_{\Sigma_X}({z}, X)$, or $\Sigma_X^\dagger$ in \eqref{true:eq}, must be estimated from the (noisy) data. 
For simplicity, we assume below that ${\mu}$ and $\sigma$ are
known; these assumptions can easily be removed in applications to real data. 
As discussed in Section \ref{sec:precision}, the notable challenge in estimating the \MD~from local samples of $Y$ is the estimation of the precision matrix via the pseudo-inverse, stemming from the interaction of the small eigenvalues of $\Sigma_X$ and the noise.
In addition, as discussed in Section \ref{sec:precision}, the main difference between \eqref{SpikeModel:X} and \eqref{ManifoldModel:X} is that in \eqref{SpikeModel:X}, the covariance matrix $\Sigma_X$ is global and strictly of low rank, whereas in \eqref{ManifoldModel:X}, the covariance matrix $\Sigma_x$ depends on $x$, and its rank is in general higher than $d$ when the manifold is of dimension $d$, yet only its first $d$ eigenvalues are dominant. 
In the context of \MD, this rank difference might lead to unwanted consequences if we consider Definition \ref{Definition:MD under linear spike}. Below we demonstrate it with an example. 

Assume for simplicity that the pdf is uniform, $\mathcal{O}(x)=B_{\epsilon}^{\mathbb{R}^p}(\iota(x))$ and $\epsilon$ is sufficiently small.
From \eqref{Expansion:LocalCov}, we observe a clear spectral gap between the first $d$ eigenvalues and the remaining ones. Therefore, we would expect that the behavior of $X_x$ in \eqref{ManifoldModel:X} is similar to that of $X$ in \eqref{SpikeModel:X}. However, we need to be careful with the associated precision matrices, and hence, with the two definitions of \MD. Suppose that the main interest in estimating the \MD~is recovering the geodesic distance when the manifold cannot be directly accessed. It has been shown in \cite[Theorem 8 (3)]{2018arXiv180402811M} that even if the manifold can be directly accessed, defining the \MD~between a point $z\in \iota(M)$ and $X_x$ by $\tilde{d}^2_{\Sigma_x}(z, X_x)=(z - x)^\top  \Sigma_x^{\dagger} (z -x)$ as in \eqref{true:eq} might lead to a biased estimate of the geodesic distance. 
To be more precise, when the rank of $\Sigma_x$ is greater than $d$ and $\|x-z\|_{\mathbb{R}^p}=t$ is sufficiently small, asymptotically we have
\begin{align}\label{ManifoldModel:LocalCovariance:Ballbad}
\frac{d(d+2)}{|S^{d-1}|\epsilon^{d+2}} \tilde{d}_{\Sigma^\dagger_x}(z, X_x)=t+O(t)\,,
\end{align}
which means that the \MD~cannot even recover the basic geodesic distance since the order of the error is of the same order as the geodesic distance.
Broadly, this bias is the result of the interaction of the small eigenvalues related to the curvature{, particularly the second fundamental form,} and the fact that the vector $z - x$ is not intrinsic to the manifold {and might contain component that is normal to the tangent space at $x$}.

%



\subsection{Motivating example}
\label{sec:eig}

A motivating example for the centrality of the \MD~under the manifold model in data analysis was presented in \cite{singer2008non}, and later elaborated in the empirical intrinsic geometry (EIG) framework presented in \cite{talmon2013empirical,talmon2015intrinsic} with various applications, e.g., \cite{talmon2011parametrization,dsilva2013nonlinear,wu2015assess,dsilva2016data,dov2016kernel,liu2018diffuse,lahav2018mahalanobis}.

Suppose the manifold $M$ is merely an image of a different, intrinsic, and inaccessible manifold of interest $N$, that is $M = \phi(N)$, where $\phi$ is a diffeomorphic map. {We may call $N$ the latent space.} In applications, $\phi$ could represent the distortion of the data introduced by some measurement equipment. This ``distortion model'' is often of key important since it critically affects the intrinsic information {$N$} we have interest in{, even if noise does not exist}.
Under this model, it is shown in \cite[Section 3]{singer2008non} that when $M$ and $N$ are both Euclidean spaces, for $x,\bar{x} \in M$, the following holds:
\begin{equation}\label{eq:EIG_dis}
\|{\theta}-\bar{\theta}\|_2^{2}=\frac{1}{2}(x-\bar{x})^{\top}
\big[C^{\dagger}+\bar{C}^{\dagger}\big](x-\bar{x})+O(\|x-\bar{x}\|_2^{4})\,,
\end{equation}
where $\theta = \phi^{-1}(x),\bar{\theta} = \phi^{-1}(\bar{x})$ and $C = \nabla\phi|_{\theta}\nabla\phi|_{\theta}^\top, \bar{C} = \nabla\phi|_{\bar{\theta}}\nabla\phi|_{\bar{\theta}}^\top$.
A similar statement for the case when $M$ and $N$ are both manifolds and $\phi$ is a diffeomorphism is given in \cite{2018arXiv180402811M}. It was further shown in \cite{singer2008non} and \cite{2018arXiv180402811M} that if we have i.i.d. sampled from $M$ adhering to the statistical manifold models discussed above, then 
\begin{equation}\label{eq:EIG_cov}
	\Sigma_x = \nabla\phi|_{\theta} \nabla\phi|_{\theta}^{\top}
\end{equation}
for a sufficiently small neighborhood $\mathcal{O}(x)$.
Remarkably, by combining \eqref{eq:EIG_dis} and \eqref{eq:EIG_cov}, a small variant of the \MD~can recover the distance between ${\theta}$ and $\bar{\theta}$ from the hidden intrinsic manifold $M_{\theta}$ based on samples from $M$ without explicit information on the map $\phi$, thereby solving a completely blind inverse problem.

The above distortion model was studied in \cite{talmon2013empirical} in the context of a nonlinear dynamical system.
Denote the dataset or the point cloud as $\mathcal{U}:=\{u_j\}_{j=1}^n\subset \mathbb{R}^{q}$, where $u_j$ is sampled at the $j$-th time stamp. The key assumption is that $\mathcal{U}$ comes from observing
an inaccessible intrinsic dynamics $\theta(t)\in \mathbb{R}^p$ that satisfies the stochastic differential equation
\begin{equation}\label{SDE_theta}
d\theta(t)=a(\theta(t))dt+d\omega(t)\,,
\end{equation}
where $a$ is an unknown drift function and $\omega$ is the standard $d$-dimensional Brownian motion. We call the subset of $\mathbb{R}^p$ that hosts $\theta(t)$ the {\em phase space}; for example, an open subset of $\mathbb{R}^p$, or a smooth manifold embedded in $\mathbb{R}^p$. 

The observation is modeled by a diffeomorphic function $\Phi: \mathbb{R}^{p}\rightarrow \mathbb{R}^{q}$ so that $u_j=\Phi(\theta_j)$, where $\theta_j$ is sampled from the intrinsic dynamics $\theta(t)$ at the $j$-th time stamp. We call $\Phi$ the {\em observation transform}.
Based on (\ref{SDE_theta}), and the fact that
\begin{equation}\label{SDE_u}
\mathrm{d}{u}_{t}=\left(\frac{1}{2}\Delta\Phi|_{\theta_{t}}+\nabla\Phi|_{\theta_{t}}a(\theta_{t})\right)\mathrm{d}t+\nabla\Phi|_{\theta_{t}}\mathrm{d}\omega_{t}
\end{equation}
by the Ito's formula, we obtain that
\begin{equation}\label{Cov vs Jacobian outer product}
\text{Cov}\left(\mathrm{d}{u}_{t}\right)=\nabla\Phi|_{\theta_{t}}\nabla\Phi|_{\theta_{t}}^\top
\end{equation}
since $(\frac{1}{2}\Delta\Phi|_{\theta_{t}}+\nabla\Phi|_{\theta_{t}}a(\theta_{t}))\mathrm{d}t$ is the drift.
By combining the above facts, 
it is shown in \cite[Section 3]{singer2008non} that when $u_i$ and $u_j$ are sufficiently close and the phase space is flat, we could recover the intrinsic distance between {${\theta}_i$ and ${\theta}_j$} by
\begin{equation}\label{EIG_dis}
\|{\theta}_i-{\theta}_j\|_{\mathbb{R}^{d}}^{2}=\frac{1}{2}(u_i-u_j)^{\top}
\big[C_i^{\dagger}+C_j^{\dagger}\big](u_i-u_j)+O(\|u_i-u_j\|^{4})\,,
\end{equation}
where ${C}_i=\nabla \Phi({\theta}_i)[\nabla \Phi({\theta}_i)]^{\top}$ is the covariance matrix associated with the observation process (i.e., the deformed Brownian motion). Furthermore, it is shown in \cite{singer2008non,talmon2013empirical} that ${C}_i$ can be estimated by the covariance matrix of $\{u_k\}_{k=i-L}^{i+L}$, where $L\in \mathbb{N}$ is chosen by the user. The key relevant fact here is that since the covariance matrix of $\mathrm d u_t$ is $\nabla\Phi|_{\theta_{t}}\nabla\Phi|_{\theta_{t}}^\top$, the related open set $O(u_t)$ is an ellipsoid with principle semi-axes described by the non-degenerate eigenvalues and eigenvectors of $\nabla\Phi|_{\theta_{t}}\nabla\Phi|_{\theta_{t}}^\top$.

To conclude, the intrinsic signal model in \eqref{SDE_theta} and the observation model \eqref{SDE_u} are generic and can describe a broad range of applications. 
Therefore, the ability to recover the distances between samples of the intrinsic process from observations in a nonparametric and unsupervised manner using an estimate of the Mahalanobis distance is very powerful.
Indeed, this model was used for system identification \cite{talmon2011parametrization}, molecular dynamics \cite{dsilva2013nonlinear}, sleep analysis \cite{wu2015assess,liu2018diffuse}, model reduction \cite{dsilva2016data}, speech processing \cite{dov2016kernel}, and gene expression data \cite{lahav2018mahalanobis}.

\subsection{Shrinkage Estimators}

For any $p$-by-$p$ matrix $M_n$ estimated from $y_1,\ldots,y_n$,
consider the estimator for \MD
\begin{equation} \label{M:eq}
    d^2_{M_n}(z, X) =  (z-\V{\mu})^\top  M_n (z-\V{\mu})\,.
\end{equation}
using Definition \ref{Definition:MD under linear spike}. The extension to Definition \ref{Definition:MD:manifoldModel} is straight-forward.
In order to quantitatively measure the performance of any \MD~estimator $d_{M_n}(z,X)$, it is useful 
to introduce a loss function.
For any estimator of the form \eqref{M:eq}, the 
absolute value of the estimation
error with respect to the true value \eqref{true:eq} is
\begin{align*}
&\Big| d^2_{\Sigma_X}(z, X) - d^2_{M_n}(z, X)) \Big | 
= \Big| (z-\V{\mu})^\top  [\Sigma_X^\dagger - M_n] (z - \V{\mu}) \Big|\,.
\end{align*}
As the test vector $z$ is arbitrary, it is natural to consider the {\em worst} case, and define
the loss of $M_n$ at the (unknown) underlying low-dimensional covariance
$\Sigma_X$:
\begin{defn}
The worst case loss function of an estimator of the form \eqref{M:eq} for \eqref{true:eq} is defined as 
\begin{align} \label{L:eq}
    L_n(M_n,\, \Sigma^\dagger_X) &:= \sup_{{z:\,\norm{z-\mu}_{\mathbb{R}^p}=1}}
	\left|(z-\mu)^\top  [\Sigma_X^\dagger - M_n] (z - \mu)\right| = \|\Sigma_X^\dagger - M_n\|_{\textmd{op}}\,,
\end{align}
where $\norm{\cdot}_{\textmd{op}}$ is the matrix operator norm. 
\end{defn}
It is also reasonable to consider the root mean squared estimation error of all possible test vectors. The discussion below follows the same line. 
To keep the notation light, the dependence of $L_n$ on $\mu$ and $\sigma$ as well as the dependence of $M_n$ on the sample $y_1,\ldots,y_n$ are implicit.

Consider matrices of the form $M^\eta_n:=\eta(S_n)$, where $\eta:[0,\infty)\to[0,\infty)$ and $S_n$ is the sample covariance. We call $M^\eta_n$ the {\em shrinkage estimator} of $\Sigma_X^\dagger$ with $\eta$. A typical example is the {\em classical \MD~estimator}, which is a shrinkage estimator with $\eta=\eta^\textup{classical}_\sigma$, where
\begin{equation}\label{eta_classic:eq}
	\eta_\sigma^{\textup{classical}}(\alpha) = \left\{ \begin{array}{cc}
		1/(\alpha-\sigma^2) & \alpha > \sigma^2 \\
		0 		& \alpha \le \sigma^2
	\end{array} \right..
\end{equation}
From \cite{mahalanobis1936generalized}, in the traditional setup when the dimension $p$ is fixed and $n \rightarrow \infty$, the classical \MD~
estimator obtains zero loss asymptotically.
\begin{thm}\label{thm1}
    Let $p$ be fixed independently of $n$.
    Then
\[
    \lim_{n\to\infty} L_n(\eta^\textup{classical}_\sigma(S_n),\, \Sigma^\dagger_X)=0\,.
\]
\end{thm}
\begin{proof} 
Since it is well known that $(S_n-\sigma^2 I_p)\to\Sigma_X$ as $n \to \infty$, substituting $M_n$ with $\eta^\textup{classical}_{\sigma}(S_n)$ in \eqref{L:eq} and taking limit with $n \to \infty$ complete the proof. 
\end{proof}
%

When $p$ grows with $n$, such
that $p = p_n\to\infty$ with $p_n/n\to \beta>0$, the situation is quite different. 
It is known that in this situation
the sample covariance matrix is an inconsistent estimate of the population covariance matrix \cite{Johnstone2006}, and Theorem~\ref{thm1} might not hold; that is, the classical \MD~estimator might not be optimal. 
The following questions naturally arise when $\beta>0$:
\begin{enumerate}
    \item 
        Is there an optimal shrinkage (OS) estimator with respect to the loss $L_n$?
    \item 
        How does the {loss of the optimal shrinkage estimator} compare with the loss $L_n(\eta^\textup{classical}_\sigma(S_n),\, \Sigma_X^\dagger)$?
\end{enumerate}
In the sequel, we attempt to answer these questions.

\section{OPTIMAL RECOVERY OF {PRECISION MATRIX FOR MAHALANOBIS DISTANCE} {UNDER THE SPIKED MODEL}}

We start the derivation of the OS for \MD~under the linear spiked model{, which involves the OS for the precision matrix}. Its extension to the manifold model requires only an additional mild condition, which will be discussed in Section \ref{Sec:manifoldModelOptShr}. Without loss of generality, we set the noise level $\sigma = 1$ and will discuss the general case subsequently.

\begin{assum}[Asymptotic($\beta$)]\label{assum1}
    The number of variables $p = p_n$ grows with the number of observations $n$,
    such that $p/n \to \beta$ as $n \to \infty$, for $0 < \beta \leq 1$.
\end{assum}

\begin{assum}[Spiked model]\label{assum2}
Suppose $\Sigma_X = \Sigma_Y-\sigma^2 I_p$ with the eigendecompostion: 
\begin{equation}\label{SigmaX_low_rank:eq}
\Sigma_X = U \begin{bmatrix} \Sigma_d & 0 \\ 0 & 0_{p-d} \end{bmatrix} U^\top \in
    \mathbb{R}^{p \times p}\,,
\end{equation}
where $d\geq 0$, $\Sigma_d=\textup{diag}(\ell_1,\cdots \ell_d)$ is a $d \times d$ matrix whose diagonal consists of $d$ \emph{spikes} $\ell_1 > \cdots > \ell_d > 0$, which are fixed and independent of $p$ and $n$, and the off-diagonal elements are set to zero. For completeness, denote $\ell_{d+1}= \ldots= \ell_p = 0$. Note that we assume that all spikes are simple. When $d=0$, it is the null case.
\end{assum}

Denote the eigendecompostion of $S_n$ as
\begin{equation}\label{Sn_evd:eq}
S_n = V_n \textup{diag}({\lambda_{1,n}, \ldots, \lambda_{p,n}}) V_n^\top \in
    \mathbb{R}^{p \times p},
\end{equation}
where $\lambda_{1,n} \geq \ldots \lambda_{p,n} \geq 0$ are the empirical eigenvalues and $V_n \in O(p)$ is the matrix, whose columns are the empirical eigenvectors $v_{i,n}\in \mathbb{R}^p$, $i=1,\ldots,p$.
Under Assumption \ref{assum1} and Assumption \ref{assum2}, results collected from
\cite{0025-5734-1-4-A01, 2004math......3022B, BAIK20061382, Paul2007} imply three important facts about the sample covariance matrix $S_n$.
\begin{enumerate}

\item 
{\em Eigenvalue spread.}
Suppose Assumption~\ref{assum1} holds and consider the null case where $\Sigma_d = 0$. As $n \to \infty$, the spread of the empirical eigenvalues $\lambda_{i,n}$ converges to a continuous distribution called the ``Marcenko-Pastur'' law \cite{0025-5734-1-4-A01}, 
\begin{equation}\label{eq:mp}
\frac{\sqrt{(\lambda_+-x)(x-\lambda_-)}}{2\pi\beta x}\mathbf{1}_{[\lambda_-, \lambda_+]}dx\,,
\end{equation}
where $\lambda_+ = (1+\sqrt{\beta})^2$ and $\lambda_- = (1-\sqrt{\beta})^2$ are the {\em limiting bulk edges}.

\item
{\em Top eigenvalue bias.}
Suppose Assumption~\ref{assum1} and Assumption \ref{assum2} hold.  For $1 \leq i \leq d$, the empirical eigenvalues 
\[
\lambda_{i,n} \xrightarrow{a.s.}\lambda(\ell_i) =: \lambda_i
\] 
as $n \to \infty$ , where
\begin{equation}\label{lambda:eq }
\lambda(\alpha) = \begin{cases}
				1+\alpha+\beta+\frac{\beta}{\alpha} &  \alpha>\ell_+\\
			 							 (1+\sqrt{\beta})^2  &  0 \leq \alpha \leq \ell_+
			 						\end{cases}
\end{equation}
is defined on $\alpha \in [0,\infty)$ and $\ell_+:=\sqrt{\beta}$. For $d+1\leq i \leq p$, since $\ell_i = 0$ the empirical eigenvalues $\lambda_{i,n}$ follow the Marcenko-Pastur law \eqref{eq:mp}.  

\item
{\em Top eigenvector inconsistency.}
Suppose Assumption~\ref{assum1} and Assumption \ref{assum2} hold. Let $c_{i,n}$ and $s_{i,n}$ be the cosine and sine values of the angle between the $i$-th population eigenvector and the $i$-th empirical eigenvector after properly adjusting the sign of each empirical eigenvector. Note that there exists a sequence of $R_{n} \in O(p)$ so that $R_{n}V_{n}$ converges almost surely (a.s.) to $V\in O(p)$. In the following we assume that the empirical eigenvectors have been properly rotated. It is known that when $n \to \infty$, $c_{i,n} \xrightarrow{a.s.} c(\ell_i)$ and $s_{i,n} \xrightarrow{a.s.} s(\ell_i)$, where
\begin{equation}\label{cosine:eq}
    c(\alpha) :=
    \begin{cases}
        \sqrt{\frac{\alpha^2-\beta}{\alpha^2+\beta\alpha}} & \alpha>\ell_+ \\
        0 & 0 \leq \alpha\leq \ell_+\,,
    \end{cases}
\end{equation}
and
\begin{equation}\label{sine:eq}
s(\alpha):=\sqrt{1-c^2(\alpha)}\,
\end{equation}
are defined on $\alpha \in [0,\infty)$.

\end{enumerate}

The above three properties imply that the classical estimator $\eta_\sigma^{\textup{classical}}(S_n)$ may not be the best estimator in general, and for the purpose of estimating \MD~in particular. Inspired by \cite{2013arXiv1311.0851D}, we may ``correct'' the bias of the eigenvalues to improve the estimation. 
\begin{defn}
The {\em asymptotic loss function} is defined as
\begin{equation}\label{loss_infty:eq}
L_\infty(\eta|\ell_1, \ldots, \ell_d) := \lim_{n \to \infty} L_n(M_n^\eta,\,  \Sigma_X^\dagger)\,, 
\end{equation}
assuming the limit exists. 
\end{defn}
To find a shrinkage estimator $\eta$ that minimizes $L_{\infty}(\eta|\ell_1, \ldots, \ell_d)$, it is natural to construct the estimator by recovering the spikes $\ell_i$ using the biased eigenvalues $\lambda_i$. From the inversion of \eqref{lambda:eq }, recalling that $\ell_{+}=\sqrt{\beta}$, we can define 
\begin{equation}\label{ell:eq}
    \ell(\alpha) :=
    \frac{\alpha+1-\beta+\sqrt{(\alpha+1-\beta)^2-4\alpha}}{2} - 1 
\end{equation} 
when $\alpha>\lambda_+$, and consider the shrinkage function  
\begin{equation}\label{eta_inv:definition}
\eta^{\textup{inv}}(\alpha) = 
    \begin{cases}
        1/\ell(\alpha) & \alpha>\lambda_+ \\
        0 & \textrm{otherwise}.
    \end{cases}
\end{equation}
{Note that since taking inverse of a matrix is nonlinear, it is not clear if the above naive idea will lead to the OS estimator of the precision matrix for the \MD~estimate. Nevertheless,}
it is reasonable to expect the existence of an optimal shrinkage function $\eta^*$ satisfying 
\begin{equation*}
L_\infty(\eta^*|\ell_1, \ldots, \ell_d) \leq L_\infty(\eta^{\textup{inv}}|\ell_1, \ldots, \ell_d)
\end{equation*}
for any spikes $\ell_1, \ldots, \ell_d$. Below we show that {this naive idea,} $\eta^{\textup{inv}}$, is in fact the {OS estimator for our purpose}. 
 
\subsection{Derivation of the Optimal Shrinker when $\sigma=1$}\label{Section:OS:linearSpike}

\begin{defn}
A function $\eta: [0, \infty) \to [0, \infty)$ is called a {\em shrinker} if it is continuous when $\lambda>\lambda_+$, and $\eta(\lambda) = 0$ when $0\leq \lambda \leq \lambda_+$. 
\end{defn}
Note that this shrinker is a {\em bulk shrinker} considered in \cite[Definition 3]{2013arXiv1311.0851D}. Based on the assumption of a shrinker $\eta$, the associated shrinkage estimator converges almost surely, that is      
\begin{equation}
M^\eta_n  \xrightarrow{a.s.} M^\eta := V  
\textup{diag}({\eta(\lambda_1), \ldots, \eta(\lambda_p)})  V^\top\,, 
\end{equation}
where the right hand side is the eigendecomposition of $M^\eta$.
Thus, the sequence of loss functions also almost surely converges as
\begin{equation}\label{loss_conv:eq}
    L_n(M^\eta_n,\,  \Sigma_X^\dagger) = \|\Sigma_X^\dagger - M^\eta_n\|_{op} \xrightarrow{a.s.} \|\Sigma_X^\dagger - M^\eta\|_{op}\,.
\end{equation}
As a result, the limit in \eqref{loss_infty:eq} exists when $\eta$ is a shrinker, and we have the following theorem which in turn gives rise to the optimal shrinker. {Note that while the biased eigenvalues could be recovered by the quadratic relationship between the biased eigenvalues and the population eigenvalues \eqref{lambda:eq }, true (population) eigenvectors and empirical eigenvectors are not collinear \cite{2013arXiv1311.0851D} and so far we do not have a way to recover the biased eigenvectors. The main idea beyond the proof is respecting this fact, and when we find the optimal way to correct the eigenvalues, the biased eigenvectors should be taken into account. With the control of this eigenvector bias, the OS will be derived.}

\begin{thm}[Characterization of the asymptotic loss]\label{thm2}
Suppose $\sigma = 1$. Consider the spiked covariance model satisfying Assumption~\ref{assum1} and Assumption~\ref{assum2} and a shrinkage function $\eta: [0, \infty) \to [0, \infty)$. We have a.s.
\begin{equation}\label{thm2_prf1}
        L_\infty(\eta|\ell_1, \ldots, \ell_d)=
        \max_{i=1,\ldots,p} \{ \Delta(\ell_i,\eta(\lambda_i)) \} \,,
\end{equation}
where $\Delta:[0,\infty)\times [0,\infty)\to [0,\infty)$ is given by
\begin{equation}\label{loss_case:eq}
        \Delta(\alpha,\zeta)=
        \begin{cases}        
       	u_+(\alpha, \zeta)
             & \alpha>\ell_+ \mbox{ and } \zeta \leq \frac{1}{\alpha} \\
	  	-u_-(\alpha, \zeta)            
            & \alpha>\ell_+ \mbox{ and } \zeta > \frac{1}{\alpha} \\
           1/\alpha & 0 < \alpha\leq \ell_+ \\
            0 & \alpha = 0\,,
        \end{cases}
\end{equation}
where
\begin{equation}\label{u_+}
     u_+(\alpha, \zeta) = \frac{1}{2}\left( \frac{1}{\alpha}-\zeta + 
                \sqrt{
                    \Big(\frac{1}{\alpha}-\zeta\Big)^2+4\frac{\zeta s(\alpha)^2}{\alpha}
                }             
            \right)\,, 
\end{equation}

\begin{equation}\label{u_-}
      u_-(\alpha, \zeta) = \frac{1}{2}\left( \frac{1}{\alpha}-\zeta -
                \sqrt{
                    \Big(\frac{1}{\alpha}-\zeta\Big)^2+4\frac{\zeta s(\alpha)^2}{\alpha}
                }              
            \right) \,.
\end{equation}               
\end{thm}
 
\begin{proof}
Based on the property of ``simultaneous block-diagonalization'' for $\Sigma_X^{\dagger}$ and $M^\eta_n$ in \cite[Section 2]{2013arXiv1311.0851D},  the properties of ``orthogonal invariance'' and ``max-decomposability'' for the operator norm in \cite[Section 3]{2013arXiv1311.0851D}, and the convergence of $c_{i,n}$ and $s_{i,n}$ in \eqref{cosine:eq} and \eqref{loss_conv:eq}, we have 
\begin{equation*}
 L_n(M^\eta_n,\,\Sigma_X^\dagger) =  \max_{i} \norm{A_i - B_{i,n}}_{\textmd{op}}  \,,
\end{equation*}
where 
$$
A_i =  \begin{bmatrix} 1/\ell_i & 0 \\ 0 & 0 \end{bmatrix}
$$ when $\ell_i\neq 0$ and $A_i=0_{2\times 2}$ otherwise, and
$$
B_{i,n} = \eta(\lambda_{i,n}) \begin{bmatrix} c_{i,n}^2 & c_{i,n}s_{i,n} \\
c_{i,n}s_{i,n} & s_{i,n}^2\end{bmatrix}\,.
$$ 
When $n\to \infty$, the loss converges a.s. to $\max_{i} \norm{A_i - B_i}_{\textmd{op}}$, where
$$
B_i = \eta(\lambda_i) \begin{bmatrix} c(\ell_i)^2 & c(\ell_i)s(\ell_i) \\
c(\ell_i)s(\ell_i) & s(\ell_i)^2\end{bmatrix} \,.
$$ Now we evaluate $\norm{A_i - B_i}_{\textmd{op}}$ for different $\ell_i$. 

When $\ell_i >\ell_+$, denote the eigenvalues of $A_i - B_i$ as $u_+(\ell_i,\eta(\lambda_i))$ and $u_-(\ell_i, \eta(\lambda_i))$. If $\eta(\lambda_i) > 1/\ell_i$ we have $0\leq u_+(\ell_i,\eta(\lambda_i)) \leq -u_-(\ell_i,\eta(\lambda_i))$, and hence $\norm{A_i-B_i}_{op} = -u_-(\ell_i,\eta(\lambda_i))$; otherwise, we have $u_+(\ell_i,\eta(\lambda_i)) \geq -u_-(\ell_i,\eta(\lambda_i))\geq 0$, and hence $\norm{A_i-B_i}_{op} = u_+(\ell_i,\eta(\lambda_i))$.
For $0 < \ell_i \leq \ell_+$, since $c(\ell_i) = 0 $, we have  
$$
B_i = \begin{bmatrix} 0 & 0 \\  0 &  \eta(\lambda_i) \end{bmatrix}\,,$$
which equals $0_{2\times 2}$ since $\eta(\lambda_i)=0$ by the definition of shrinkage function. Thus, $\norm{A_i - B_i}_{op} = 1/\ell_i$. 
Finally, for $\ell_i = 0$, $A_i$ is a $2\times 2$ zero matrix, and thus
$\norm{A_i - B_i}_{op} = \eta(\lambda_i)=0$.  This concludes the proof.
\end{proof}

\begin{figure}[htb!]
\centering
\includegraphics[trim=20 200 20 230, clip, width=0.5\textwidth]{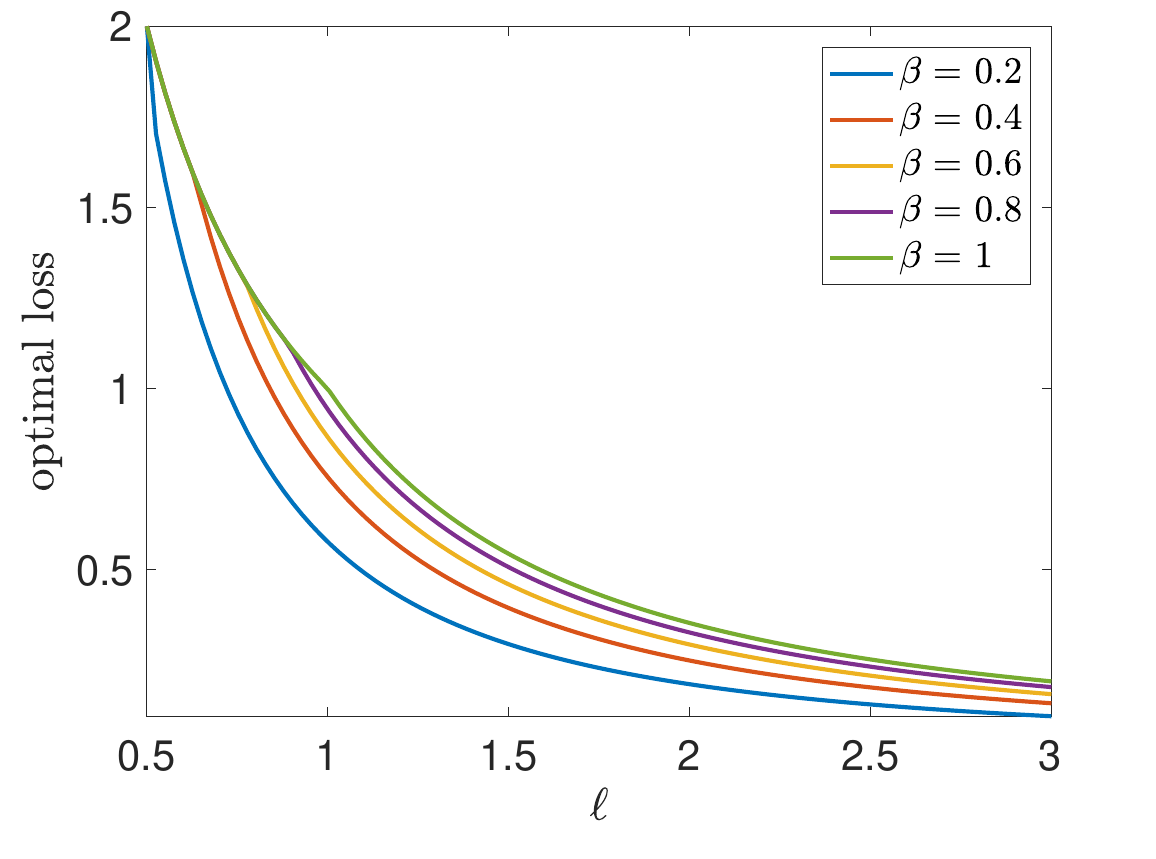}
\caption{The asymptotic loss for several $\beta = p/n$ values as a function of the spike strength in a single spike model, i.e., 
$\ell_1 \in [0.1,3]$, and $\ell_2 = \ell_3 = \ldots = \ell_d = 0$.
The noise level is fixed and set to $\sigma =1$.}
\label{fig:asym_loss}
\end{figure}

Figure \ref{fig:asym_loss} illustrates the obtained asymptotic loss for several $\beta = p/n$ values as a function of the spike strength in a single spike model.
It is clear that for each $\beta$, there is a transition at $\ell_+=\sqrt{\beta}$.
An immediate consequence of Theorem \ref{thm2} is that $\eta^{\textup{inv}}$ is an optimal shrinker.

\begin{cor}\label{Corollary:1}
Suppose $\sigma=1$ and Assumption~\ref{assum1} and Assumption~\ref{assum2} hold. Define the asymptotically optimal shrinkage function as   
\begin{equation}\label{define_op_shrinker}
\eta^* := \underset{\eta}{\arg\min} \{ L_\infty(\eta|\ell_1, \ldots, \ell_d)\}\,,
\end{equation}
where argmin is evaluated on the set of all possible shrinkage functions. Then, $\eta^*$ is unique and equals $\eta^{\textup{inv}}$ given in (\ref{eta_inv:definition}). Moreover, its associated loss is
\begin{equation}\label{lemma1:eq1}
    \max_i\left\{ \Delta(\ell(\lambda_i)) \right\}\,,
\end{equation}
where
\begin{equation}\label{op_loss_cases:eq}
    \Delta(\alpha) = \begin{cases}
        s(\alpha)/\alpha = \frac{\sqrt{\beta}}{\alpha^{3/2}}\sqrt{\frac{1+\alpha}{\beta+\alpha}} &
        \alpha>\ell_+ \\
        1/\alpha & 0<\alpha\leq \ell_+ \\
        0 & \alpha = 0\,.
    \end{cases}
\end{equation}
\end{cor}

Note that this result coincides with the findings reported in \cite{2013arXiv1311.0851D}. Precisely, it is shown in \cite[(1.12)]{2013arXiv1311.0851D} that for the operator norm, $\ell(\alpha)$ \eqref{ell:eq} is the optimal shrinkage for the covariance matrix and precision matrix. In this corollary, we show that for the Mahalanobis distance, which is related to the precision matrix, the optimal estimator is also achieved by the optimal shrinkage, taking $\ell(\alpha)$ into account. 

\begin{proof}
Based on Theorem \ref{thm2}, the optimal shrinker $\eta^*$ leads to $\underset{\eta\geq 0}{\min} \underset{i=1,\ldots,d}{\max}\{\Delta(\ell_i,\eta(\lambda_i))\}$. Note that for $j=\underset{i=1,\ldots,d}{\arg\max}\{\Delta(\ell_i,\eta(\lambda_i))\}$, the optimal shrinker achieves $\underset{\eta\geq 0}{\min}\{\Delta(\ell_j,\eta(\lambda_j))\}$.
Thus, by the same argument in \cite{2013arXiv1311.0851D}, if we could solve $\underset{\eta\geq0}{\arg\min}\{\Delta(\alpha,\eta(\lambda(\alpha)))\}$ for any $\alpha>0$, we find the optimal shrinker. To simplify the notation, we abbreviate $\eta(\lambda(\alpha))$ by $\eta$.

For  $\alpha > \ell_+$ and $\eta > \frac{1}{\alpha}$, we have $\Delta(\alpha,\eta) = -u_-(\alpha,\eta)$. By a direct calculation, we get 
\[
\partial_\eta \Delta(\alpha,\eta)  =\frac{1}{2}+\frac{-(\frac{1}{\alpha}-\eta)+\frac{2s(\alpha)}{\alpha}}{\sqrt{(\frac{1}{\alpha}-\eta)^2+4\frac{\eta s(\alpha)^2}{\alpha}}} >0. 
\]
For  $\alpha > \ell_+$ and $0\leq \eta \leq \frac{1}{\alpha}$, we have $\Delta(\alpha,\eta) =u_+(\alpha,\eta)$, and similarly by taking the derivative of \eqref{u_+} we have 
\[
\partial_\eta \Delta(\alpha,\eta)  =-\frac{1}{2}+\frac{-(\frac{1}{\alpha}-\eta)+\frac{2s(\alpha)}{\alpha}}{2\sqrt{(\frac{1}{\alpha}-\eta)^2+4\frac{\eta s(\alpha)^2}{\alpha}}} \geq 0. 
\]
As a result, the partial derivative of the loss function is decreasing when $0\leq \eta\leq 1/\alpha$ and increasing when $\eta>1/\alpha$ with a discontinuity at $\eta = 1/\alpha$ while the loss function is continuous. Thus, the loss function reaches the minimum when $\eta = 1/\alpha$. These facts imply that $\eta^*(\lambda_i) = 1/\ell(\lambda_i)$ when $\lambda_i > \ell_+$. Furthermore, by substituting $\eta$ with $\eta^*$ in \eqref{u_+} or \eqref{u_-}, we get $\Delta(\alpha) = s(\alpha)/\alpha$. 
By definition, $\eta^* = 0$ when $0 \leq \alpha \leq \ell_+$. 
Thus, for $0 <\alpha \leq \ell_+ $, $\Delta(\alpha) = 1/\alpha$,
and for $\ell = 0$, $\Delta(\ell) = 0$. 
Finally, it is clear that $\eta^*$ is continuous when $\alpha>\lambda_+$, and $\eta(\alpha) = 0$ when $0\leq \alpha \leq \lambda_+$.
We thus conclude that $\eta^*$ is the optimal shrinker.
\end{proof}

We compare our result with another naive approach; that is, obtaining the covariance by the optimal shrinkage with respect to the operator norm, and then taking the Moore-Penrose pseudo-inverse. Let $\eta^{cov}$ denote the optimal shrinker for the covariance matrix recovery obtained from \cite{2013arXiv1311.0851D} with respect to the operator norm loss function. With the same notation, that is, $\Sigma_Y = \Sigma_X+\sigma^2 I_p$ and $\sigma = 1$ as in \eqref{assum2}, the optimal shrinker satisfies 
\begin{equation}
\eta^{cov}(\alpha) = 
    \begin{cases}
        \ell(\alpha)+1 & \alpha>\lambda_+ \\
        1 & \textrm{otherwise}\,,
    \end{cases}
\end{equation}
where $\ell$ is defined in \eqref{ell:eq}.
Note that in \cite{2013arXiv1311.0851D}, the authors aimed to recover $\Sigma_Y$, while we recover $\Sigma_X^{\dagger}$. In other words, the authors in \cite[(1.12)]{2013arXiv1311.0851D} showed that when the operator norm is considered, the OS estimator is the same as correcting the eigenvalues according to the relationship \eqref{lambda:eq }.
Thus, our result $\eta^*$ coincides with the inverse of $\eta^{cov}-1$ when $\alpha>\lambda_+$.

We note the following interesting phenomenon stemming from Theorem \ref{thm2} {and Corollary \ref{Corollary:1}}. If there exists a nontrivial spike $\ell_i>0$ that is weak enough so that
{$\ell_i$ is sufficiently small compared with $\ell_+$}, then $L_\infty(\eta|\ell_1, \ldots, \ell_d)$ is dominated by
$1/\ell_i$. Consequently, in this large $p$ large $n$ regime, we cannot ``rescue'' this spike, and the associated signal is lost in the noise, {as can be seen in Corollary \ref{Corollary:1}.} 

Figure \ref{fig:shrinker} illustrates the obtained optimal shrinker with the classical shrinker overlay, for $\beta = p/n = 1$ and $\sigma=1$.
Clearly, compared with the classical shrinker, the obtained optimal shrinker truncates the eigenvalues more aggressively.

\begin{figure}[t]
\includegraphics[width=0.95\textwidth]{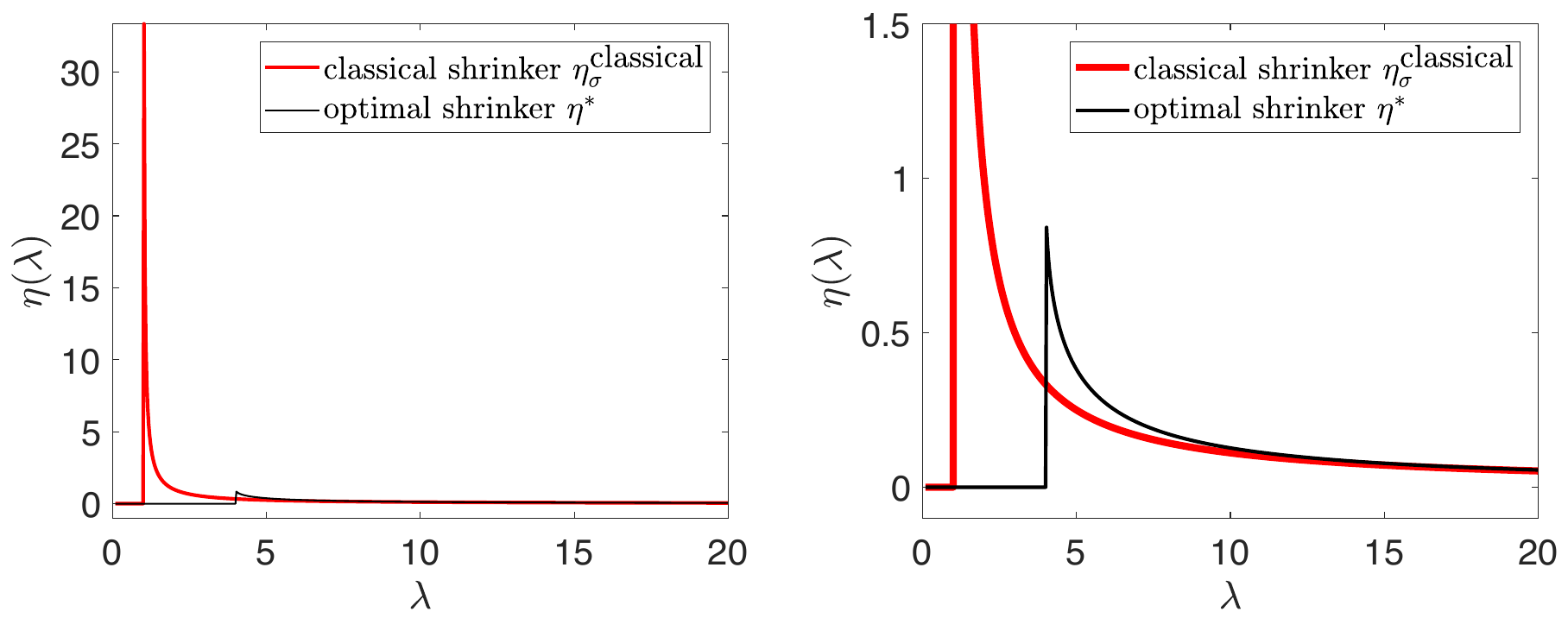}
\caption{The obtained optimal shrinker with the classical shrinker overlay, for $\beta = p/n = 1$ and $\sigma=1$. To enhance the visualization, the y-axis of the figure is truncated at $1.5$ on the right hand side.}
\label{fig:shrinker}
\end{figure}

\subsection{Derivation of the Optimal Shrinker when $\sigma\neq 1$}\label{Section:OS:sigmanot1}

To handle the general case when $\sigma \neq1$, we first rescale the data and model by setting $\ell'_i := \ell_i/\sigma^2$ and $\lambda'_{i,n} := \lambda_{i,n}/\sigma^2$, and consider the following shrinker defined on $[0,\infty)$:
\begin{equation}
\eta_\sigma(\alpha):= \frac{\eta(\alpha/\sigma^2)}{\sigma^2}\,.
\end{equation}
Note that since $\eta$ plays the role of estimating the precision matrix, we re-normalize it by dividing $\eta(\alpha/\sigma^2)$ by $\sigma^2$.
The shrinkage estimator for $\Sigma_{X}^\dagger$ becomes $M^{\eta_\sigma}_n := \eta_\sigma(S_n)$, the general optimal shrinker becomes
\begin{equation}\label{general_sh:eq}
        \tilde{\eta}^*(\alpha) = \begin{cases}
        	\frac{1}{\sigma^2 \ell(\alpha/\sigma^2)} & \alpha > \sigma^2\ell_+ \\
        	0 & 0\leq \alpha \leq \sigma^2\ell_+
        \end{cases}\,
\end{equation}
and the associated loss is
\[
    \max_i\left\{ \frac{\Delta(\ell(\frac{\lambda_i}{\sigma^2}))}{\sigma^2} \right\}\,.
\]

\section{EXTENSION TO THE MANIFOLD MODEL}\label{Sec:manifoldModelOptShr}

We now {come back to the manifold learning problem with the manifold setup described in Section \ref{Section:ManifoldModel}. We argue that despite the challenge mentioned in Section \ref{sec:precision},} the developed theorem in Section \ref{Section:OS:linearSpike} can be {extended} to study the manifold model in the large $p$ large $n$ setup {with proper modifications}. 
Note that For each $n$, there exists an orthonormal basis $\{u_{n,l}\}_{l=1}^p$ of $\mathbb{R}^p$ so that the $d$-dimensional compact smooth manifold $M$ is {isometrically} embedded in the subspace spanned by $\{u_{n,l}\}_{l=1}^K$ for a fixed $K\in \mathbb{N}$. 
In other words, while the rank of $\Sigma_x$ associated with $\mathcal{O}(x)$ depends on $x$, it is bounded uniformly from above by $K$. Note that $K$ in general can be much larger than $d$, yet it is fixed
{due to the well known Nash's isometric embedding theory \cite{nash1956imbedding}, which guarantees the existence of $K$ that is independent of $p$ and $K\leq d(3d+11)/2$.} 
We put the following assumption.

\begin{assum}\label{Assumption O ellipsoid}
Assume $\mathcal{O}(x)=\phi(B_\epsilon(0))$, where $B_\epsilon(0)$ is a Euclidean ball centered at $0$ with the radius $\epsilon>0$, and $\phi:\mathbb{R}^d\to M$ is diffeomorphic on $B_\epsilon(0)$. We call $\mathcal{O}(x)$ an ellipsoid.
\end{assum}

Note that since $\phi$ is diffeomorphic, {in general $\mathcal{O}(x)$ is not really an ellipsoid unless $\phi$ is linear. But to simplify the terminology, we abuse the notation and still call it an ellipsoid. Also note that} the elliptic radii {of $\mathcal{O}(x)$} is of order $\epsilon>0$ {with the implied constant depending on the Jacobian of $\phi$ when $\epsilon$ is sufficiently small}.

Based on the developed theorem in Sections \ref{Section:OS:linearSpike} and \ref{Section:OS:sigmanot1}, we state the following theorem that secures the recovery of \MD~under the manifold model in Definition \ref{Definition:MD:manifoldModel}.  
The basic idea beyond this theorem is twofold. First, it relies on the result from \cite[Lemma 3 and Lemma 6]{2018arXiv180402811M} stating that when $\epsilon>0$ is sufficiently small, there is a sufficiently large gap between the first $d$ eigenvalues of $\Sigma_x$ and the remaining small eigenvalues. Second, as discussed in Section \ref{Section:OS:linearSpike}, any nontrivial eigenvalue $\ell_i$ that satisfies $\ell_i\ll \ell_+$ is ignored by the optimal shrinkage. 
\begin{thm}\label{Theorem:recover MD manifold}
Assume Assumptions \ref{assum1}-\ref{Assumption O ellipsoid} hold.
Fix $x\in M$. 
Suppose $\sigma=\sigma(\epsilon)$ so that $\sigma^2\sqrt{\beta}\epsilon^{-d-2}\to 0$ and $\sigma^2\sqrt{\beta}\epsilon^{-d-4}\to \infty$ when $\epsilon\to 0$. Assume the maximal elliptic radius of $\mathcal{O}(x)$ is $m\epsilon$, where $m=m(\epsilon)\asymp 1$, and the ratio of the maximal and minimal elliptic radii is fixed for all $\epsilon>0$.
When $\epsilon$ is sufficiently small, all nontrivial eigenvalues of $\Sigma_x$ except the top $d$ eigenvalues are set to zero by the optimal shrinkage $\tilde{\eta}^*$. 
\end{thm}
\begin{proof}
By \cite[Lemma 6]{2018arXiv180402811M}, the local covariance matrix has the following asymptotical expansion when $\epsilon$ is sufficiently small:
\begin{equation}
\Sigma_x=\frac{|S^{d-1}|P(x)\epsilon^{d+2}}{d(d+2)}[\iota_*|_x\nabla \phi|_{\phi^{-1}(x)}][\iota_*|_x\nabla \phi|_{\phi^{-1}(x)}]^\top+O(\epsilon^{d+4}),
\end{equation}
where $|S^{d-1}|$ is the volume of $S^{d-1}$.
A derivation, similar to the derivation in \cite[Lemma 3]{2018arXiv180402811M}, yields that when $\epsilon$ is sufficiently small, the top $d$ eigenvalues of $\Sigma_x$ are of order $\epsilon^{d+2}$, and the remaining eigenvalues are of order equal to or higher than $\epsilon^{d+4}$. In other words, the ``signal strength'' is of order $\epsilon^{d+2}$ while the noise strength is $\sigma>0$. Note that there are at most $K$ non-zero eigenvalues.

By \eqref{general_sh:eq}, all eigenvalues smaller than $\sigma^2\ell_+=\sigma^2\sqrt{\beta}$ are eliminated by the optimal shrinkage. 
Combining the above, since $\beta>0$, $P(x)$ and $m$ are all fixed, when $\epsilon$ is sufficiently small so that $\sigma^2\sqrt{\beta}\epsilon^{-d-2}$ is sufficiently small and $\sigma^2\sqrt{\beta}\epsilon^{-d-4}$ is sufficiently large, all nontrivial eigenvalues of $\Sigma_x$ except the top $d$ eigenvalues are set to zero by the optimal shrinkage $\tilde{\eta}^*$. Indeed, by the assumption that $\sigma^2\sqrt{\beta}\epsilon^{-d-2}\to 0$ and $\sigma^2\sqrt{\beta}\epsilon^{-d-4}\to \infty$, we know that $c_1\epsilon^{d+2}\geq \sigma^2\sqrt{\beta}$ and $c_2\epsilon^{d+4}\leq \sigma^2\sqrt{\beta}$ when $\epsilon$ is sufficiently small for some universal constants $c_1$ and $c_2$. Therefore, $\lambda_l \geq \sigma^2\sqrt{\beta}$ for $l=1,\ldots,d$ and $\lambda_l \leq \sigma^2\sqrt{\beta}$ for $l=d+1,\ldots,K$, and hence we conclude the proof.
\end{proof}

{We conclude this section with several remarks on this theorem. First, the condition $\sigma^2\sqrt{\beta}\epsilon^{-d-2}\to 0$ seems to be limited since the noise level goes to zero when $\epsilon$ goes to zero. This condition is needed if we want to properly estimate the precision matrix locally within an ellipsoid over a manifold with elliptic radii of order $\epsilon$. In practice, $\epsilon$ plays the role of ``bandwidth'' which reflects the ``resolution'' of how accurate we could estimate the quantity we have interest. For example, in the motivating EIG example discussed in Section \ref{sec:eig}, we need an accurate precision matrix estimation so that the geodesic distance in the latent space can be accurately estimated. Such precision matrix via \eqref{Cov vs Jacobian outer product} is less affected by the curvature if $\epsilon$ is small.
However, if the noise level $\sigma$ is fixed, then $\epsilon$ is bounded from below, so that we have a limited accuracy when we recover the geodesic distance on the latent space, and hence the latent space itself by DM.}

{Second, this theorem also suggests that we {\em always} need the noise to attain a reasonable estimate of \MD; that is $\sigma^2\sqrt{\beta}\epsilon^{-d-4}\to \infty$.} This statement is {certainly} counterintuitive, since in the linear spiked model, noise absence is desired and beneficial. {To clarify this point, note that in general the dimension of the manifold is unknown, but it is needed so that we can define a sensible \MD~on the manifold in Definition \ref{Definition:MD:manifoldModel}. See the discussion in Section \ref{sec:Mahalanobis_dist}.
Thus, a dimension estimation is needed. This theorem mainly says that when $\sigma>0$ is sufficiently large so that} the condition in Theorem \ref{Theorem:recover MD manifold} holds, we could get the \MD~in Definition \ref{Definition:MD:manifoldModel} {even if we do not know the dimension. If the noise is not sufficiently big and we do not know the dimension, then we cannot obtain the desired \MD. If the dimension is known, or can be robustly estimated when noise exists, then this lower bound condition can be removed. The above facts stem from the complicated nonlinear interaction of nonlinear manifold structure and high dimensional noise. Finally, we mention that this result extends the EIG study in \cite{2018arXiv180402811M} under the manifold model when noise exists.} 

\section{SIMULATION STUDY} \label{sec:simulation}

To numerically compare the optimal shrinker and the classical shrinker, we set $\beta = 0.2, 0.4,\ldots, 1$ and consider the number of samples $n =300$ so that $p=\beta n$. 
For simplicity, we set $\ell_i = i$ for $i=1,\ldots,d$. We consider $d=\{1,4\}$. 
Suppose $x_i$, $i=1,\ldots,n$ are sampled i.i.d. from the random vector $\sum_{\ell=1}^d \zeta_\ell e_\ell $, where $e_\ell\in \mathbb{R}^p$ is the unit vector with $\ell$-th entry $1$, $\zeta_\ell \sim N(0,1)$ for $\ell=1,\ldots,d$, and $\zeta_\ell$ is independent of $\zeta_k$ when $\ell \neq k$. The noisy data is simulated by $y_i = Ax_i+\sigma^2\xi$, where $\xi \sim \Nc \left( 0, I_p \right)$ is the noise matrix, $\xi$ is independent of $\zeta_\ell$ and $A\in O(p)$ is randomly sampled from $O(p)$. In the simulation, we take $\sigma=0.225,0.45,\ldots,1.8$. For each $\sigma$, we repeat the experiment $200$ times and report the mean and variance of the loss $L_n$.

Figure \ref{fig:d4_oploss} shows the loss of the optimal and classical
shrinkers when $d=1$ and $d=4$ in a logarithmic scale. We observe that the loss using the classical
shrinker is significantly larger. This stems from the fact that in the large $p$ and
large $n$ regime, there are eigenvalues greater than $\sigma^2$ that are not
associated with the signal. When applying the classical shrinker \eqref{eta_classic:eq} (the Moore-Penrose pseudo-inverse), these irrelevant eigenvalues contribute significantly, leading to high loss. 
Conversely, the optimal shrinker is much more `selective' (as illustrated in Figure \ref{fig:shrinker}), associating larger eigenvalues with the noise, thereby increasing the robustness of the estimator.

\begin{figure*}
\noindent \includegraphics[trim=150 0 150 0, width=1\textwidth]{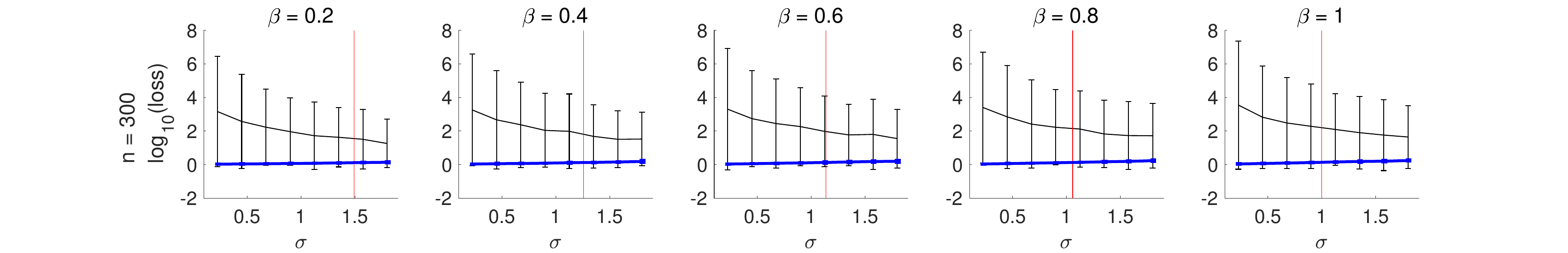}
\includegraphics[trim=150 0 150 0, width=1\textwidth]{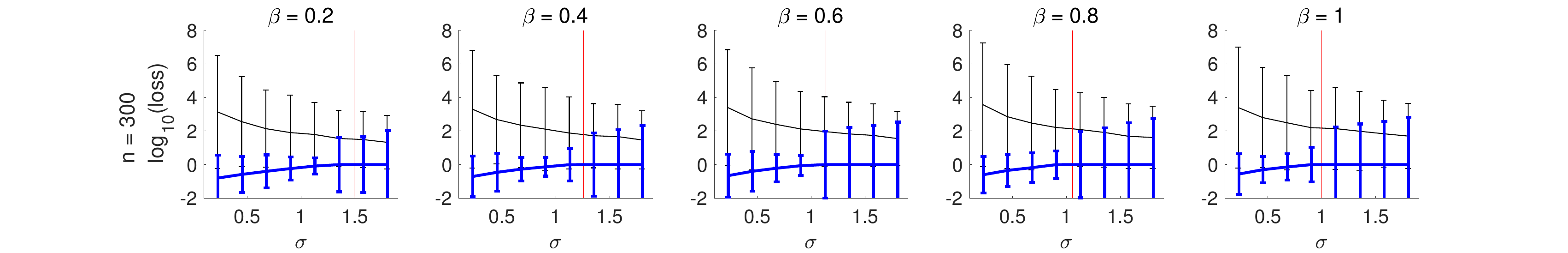}
\caption{(top) The performance of the optimal shrinker and the classical shrinker when $d=1$. (bottom) The performance of the optimal shrinker and the classical shrinker when $d=4$. The black curve represents the median of the difference between the loss of the classical shrinker and that of the theoretical optimal loss presented in log scale. The blue curve represents the median of the difference between the loss of the optimal shrinker and the theoretical optimal loss presented in log scale. The error bars depict the interquartile range of each shrinker (in log scale). The vertical blue line is $\sigma = 1/\sqrt{\ell_+}$, which indicates the tolerable noise level for the given $\beta$ and signal strength.\label{fig:d4_oploss}} 
\end{figure*}


{
Recall that our main motivation for considering \MD~in the high-dimensional regime $p\asymp n$ comes from manifold learning. 
Next, we test the performance of the proposed OS algorithm on high-dimensional data lying on a lower-dimensional manifold. 
Consider the following model 
\begin{equation}\label{eq:sim1a}
Y = X+\sigma \xi\,, 
\end{equation}
where $X$ is sampled from a curved manifold with one chart $\mathcal{M}$ embedded in $\mathbb{R}^p$ that is parametrized by 
\begin{equation}\label{eq:sim1b}
\left[s,t,4\left(\frac{s}{3}\right)^2+5\left(\frac{t}{3}\right)^2,0,\ldots,0 \right]^\top \in \mathbb{R}^p\,,
\end{equation}
$s,t \in [-5,5]$, $\xi \sim \Nc \left( 0, I_p \right) \in \mathbb{R}^p$, and $\sigma>0$ is the noise level. Note that since the precision matrix estimation and \MD~estimation are both local, we consider this one-chart manifold without loss of generality. Suppose the ambient dimension is fixed at $p = 100$. For various values of $\beta>0$, we sample $n = \lceil p/\beta\rceil$ pairs $(s,t)$ uniformly from $[-5,5]\times [-5,5]$ and generate $n$ nonuniform points on $\mathcal M$ by \eqref{eq:sim1b}. Then, these points are corrupted by additive noise with various levels $\sigma$ according to \eqref{eq:sim1a}.
The normalized loss of \MD~is computed by 
\[
	\texttt{Error}(M_n,y) := \frac{\left|d_{M_n}(y, X) - d_{\Sigma_X^{\dagger}}(y, X)\right|}{d_{\Sigma_X^{\dagger}}(y, X)},
\]
where $y\in \mathcal{M} \subset \mathbb{R}^p$ is an arbitrary point on the manifold and $M_n$ is the estimated covariance matrix.
We examine two reference points: $y_1 = [0,\ldots,0]\in \mathcal{M}$ and $y_2 = [2,2,4,0,\ldots,0] \in \mathcal{M}$. For each case, the simulation was repeated for $500$ times, and the mean and standard deviation of errors are reported.
In Table \ref{table}, we compare the performance of the OS estimator, $\tilde{\eta}^*(S_n)$, where $S_n$ is the sample covariance, to the performance of the classical estimator $\eta^{\text{classical}}_{\sigma}(S_n)$. We observe that the OS outperforms the classical estimator in this well-controlled manifold setup, {and we could see that the larger the noise, the worse the performance is. Moreover, the higher the dimension is, that is, the larger the $\beta$ is, the worse the performance is}. 

\section{{APPLICATION TO DYNAMICAL SYSTEM ANALYSIS}} \label{sec:simulation2}
Next, we apply our OS estimator on the multiscale reduction problem studied in \cite{dsilva2016data}. We consider the following two-dimensional stochastic differential equation (SDE)
\begin{equation} \label{eq:specific_SDE}
\begin{aligned}
dx_1(t) &=& 3dt &+& dW_1(t)\\
dx_2(t) &=& -\frac{x_2(t)}{\epsilon} dt &+& \frac{1}{\sqrt{\epsilon}} dW_2(t)\,,
\end{aligned}
\end{equation}
where $W_1$ and $W_2$ are independent standard Bronwian motion and $\epsilon>0$ is a small constant quantifying the scale of $x_2$. This SDE defines a dynamical system with two time scales, where $x_1$ is a slow variable and $x_2$ is a fast variable.
Suppose the state of the system $(x_1(t),\,x_2(t))$ is hidden and we have access to it through an embedding in a high dimension space via the map
\begin{equation} \label{eq:nonlinear_function}
\begin{aligned}
\mathbf{y}(t) = 
\begin{bmatrix}
y_1(t) \\ y_2(t) \\ y_3(t) \\ \vdots \\ y_p(t)
\end{bmatrix} &=&
\begin{bmatrix}
f_1(x_1(t),x_2(t))  \\
f_2(x_1(t),x_2(t)) \\ 0 \\ \vdots \\ 0
\end{bmatrix}\\
\end{aligned}.
\end{equation}
In addition, suppose the high-dimensional observation $\mathbf{y}(t)$ is contaminated by noise:
\begin{equation} \label{eq:embedding}
\begin{aligned}
\mathbf{z}(t) = 
\begin{bmatrix}
z_1(t) \\ z_2(t) \\ z_3(t) \\ \vdots \\ z_p(t)
\end{bmatrix}
&=&  \mathbf{y}(t) + 
\begin{bmatrix}
w_1(t) \\
w_2(t)\\ w_3(t) \\ \vdots \\ w_p(t)
\end{bmatrix}\,,
\end{aligned}
\end{equation}
\begin{equation}
\end{equation}
where 
\[
    dw_i(t) = \sigma * d\Omega_i(t),
\]
$\Omega_i(t)$ are independent standard Brownian motion, and $\sigma$ is the noise level.

In \cite{dsilva2016data}, diffusion maps \cite{coifman2006diffusion} with a Gaussian kernel based on the MD \cite{singer2008non} was applied to effectively reduce the dimensionality of the system by attenuating the fast variable and recovering the slow variable. Specifically, the following kernel matrix $\mathbf{W} \in \mathbb{R}^{N \times N}$ was used:
\begin{equation} \label{eq:dmaps_kernel}
W_{ij} = \exp \left( -\frac{\|\mathbf{y}(t_i) - \mathbf{y}(t_j) \|_{M}^2}{\sigma_{kernel}^2} \right)\,,
\end{equation}
where $\| \cdot \|_{M}$ is defined by
\begin{equation} \label{eq:mahalanobis_distance}
 \| \mathbf{y}(t_i) - \mathbf{y}(t_j) \|^2_M :=
 \frac{1}{2} (\mathbf{y}(t_i) - \mathbf{y}(t_j))^T \left( \mathbf{C}^{\dagger}(\mathbf{y}(t_i)) + \mathbf{C}^{\dagger}(\mathbf{y}(t_j)) \right) (\mathbf{y}(t_i) - \mathbf{y}(t_j)),
 \end{equation}
where $\mathbf{C}(\mathbf{y}(t))$ is the local population covariance of the observed stochastic process {\em at the point} $\mathbf{y}(t)$, and $\sigma_{kernel} $ is the selected kernel scale. 

We generate $N$ samples of the noisy observation $\mathbf{z}(t_i)$, where $t_i:=idt$ and $i=1,\ldots,N$. At each sampling point $t_i$, we generate a trajectory consisting of $q$ samples of noisy observations with sampling time interval $\delta t$ starting at $\mathbf z(t_i)$, such that we have $q$ samples: $\mathbf{z}(t'_1), \mathbf{z}(t'_2 ),. . . , \mathbf{z}(t'_q)$ drawn approximately from the local distribution at $\mathbf{z}(t_i)$. 
Here, we set $ \delta t = o(\epsilon)$, where $\epsilon>0$ is from \eqref{eq:specific_SDE}.
Let $\mathbf{\bar{z}}$ denote the local sample mean of $\mathbf{z}(t'_1),\ldots,\mathbf{z}(t'_q)$. The local sample covariance is given by 
\begin{equation}\label{eq_localcov}
\hat{\mathbf{C}}(\mathbf{y}(t_i)) = 
\frac{1}{\delta t(q-1)}\sum_{j = 1}^q (\mathbf{z}(t'_j)-\bar{\mathbf{z}})(\mathbf{z}(t'_j)-\bar{\mathbf{z}})^{\T}.
\end{equation} 
For the estimation of the local precision at $\mathbf{y}(t_i)$, that is, $\mathbf{C}^\dagger(\mathbf{y}(t_i))$ from the noisy samples $\mathbf{z}(t)$, we test two methods. 
The first method is based simply on the pseudo-inverse of $\hat{\mathbf{C}}(\mathbf{y}(t_i))$. 
The second method is based on the OS estimator $\eta_{\sigma}$ that is applied to \eqref{eq_localcov} to estimate the local precision estimation. 
Once the local precision matrices are estimated, we construct the following distance matrix $\mathbf{D}$, whose elements are given by
\begin{equation} \label{eq:mahalanobis_distance estimate}
 \mathbf{D}_{i,j} = \frac{1}{2} (\mathbf{z}(t_i) - \mathbf{z}(t_j))^T \left( \hat{\mathbf{C}}^{\dagger}(\mathbf{y}(t_i)) + \hat{\mathbf{C}}^{\dagger}(\mathbf{y}(t_j)) \right) (\mathbf{z}(t_i) - \mathbf{z}(t_j)),
 \end{equation}
 
We examine three cases of observation functions $f_1$ and $f_2$:
\paragraph{Case I.} Let
\begin{equation}\label{eq_embed1}
f_1(x_1(t),x_2(t)) = x_1(t), \quad f_2(x_1(t),x_2(t)) = x_2(t).
\end{equation}
and set $\epsilon = 10^{-3}$, $dt = 10^{-4}$, $\delta t= 10^{-7}$, $N = 3000$, $p = 50$, $q = 50$, $\sigma = 0.1$, and $\sigma_{kernel}$ to the $20$ percent quantile of the distance matrix $\mathbf{D}$ acquired from \eqref{eq:mahalanobis_distance estimate}. These parameters are taken from \cite{dsilva2016data}.

\paragraph{Case II.}
Let
\begin{equation}\label{eq_embed2}
f_1(x_1(t),x_2(t)) = x_1(t)+2x_2(t), \quad f_2(x_1(t),x_2(t)) = x_2(t),
\end{equation}
and set the parameters as in Case I.

\paragraph{Case III.}
Let
\begin{equation}\label{eq_embed3}
f_1(x_1(t),x_2(t)) = x_1(t)+x_2^2(t), \quad f_2(x_1(t),x_2(t)) = x_2(t).
\end{equation}
We set the parameters as in Case I and II, except for $\delta t$, which is set to a finer value $10^{-10}$, and $\sigma_{kernel}$, which is set to $5$ percent quantile of the distance matrix $\bf{D}$. These observations functions create a ``half-moon'' shape, and where considered in \cite{dsilva2016data}.

In Fig. \ref{plot_f1}, we compare the results obtained for Case I by diffusion maps based on three distance matrices. The first distance matrix is a variant of  \eqref{eq:mahalanobis_distance estimate}, where the noisy signal $\mathbf{z}(t)$ is replaced by the clean signal $\mathbf{y}(t)$. These results are plotted in the leftmost column. The second distance matrix is \eqref{eq:mahalanobis_distance estimate}, where the estimates of the precision matrix are based on the pseudo-inverse. These results are plotted in the middle column. The third distance matrix is \eqref{eq:mahalanobis_distance estimate}, where the estimates of the precision matrix are based on the proposed OS. These results are plotted in the rightmost column.
In the first row, we show a scatter plot of the slow variable $x_1$ against its most correlated eigenvector of the transition matrix associated with the diffusion map. In the second row, we show the scatter plot of the fast variable $x_2$ against its most correlated eigenvector of the transition matrix associated with the diffusion map. 
The correlation between the respective variables and eigenvectors is shown in red above each subplot. 
In the third row, we present the eigenvalues of the transition matrix associated with the diffusion map, where $d_k$ is the $k$-th largest eigenvalue.
The eigenvalues corresponding to the eigenvectors presented in the first and second rows (corresponding to the slow and fast variables) are marked by red and blue circles, respectively.
Fig. \ref{plot_f2} and Fig. \ref{plot_f3} are the same as Fig. \ref{plot_f1}, but for Case II and Case III.

We observe that all three figures present consistent results and trends.
Focusing first on the leftmost column, we see that the use of MD in diffusion maps based on the clean signal recovers accurately the slow variable $x_1$ and attenuates the fast variable $x_2$, conforming to the results presented in \cite{dsilva2016data} for only a 2-dimensional observations. 
In the middle column, we see that the addition of noise hinders both the recovery of the slow variable as well as the attenuation of the fast variable, when the pseudo-inverse is used for the estimation of the MD.
In contrast, in the rightmost column, we see that when the proposed OS is used, the obtained results based on the noisy signal are comparable to the results obtained based on the clean signal in the leftmost column.

These results demonstrate, in the context of manifold learning, that the estimation of the MD based on the pseudo-inverse fails in the high-dimensional regime with the presence of noise. In addition, they show that using the proposed optimal shrinker indeed offers a remedy and gives rise to accurate manifold learning. 
}

\begin{figure}[hbt!]
\center
\includegraphics[trim=50 0 50 0, width=.7\textwidth]{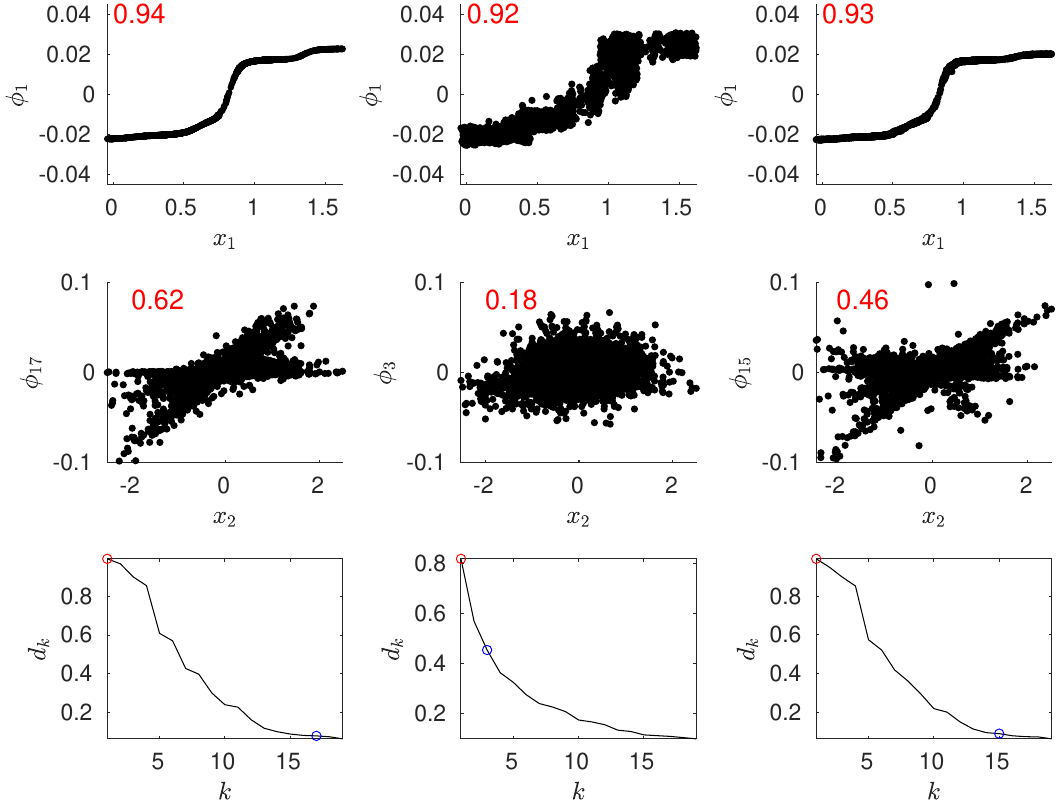}
\caption{The results obtained for Case I by diffusion maps with different distances. Left column: the distance matrix is a variant of  \eqref{eq:mahalanobis_distance estimate}, where the noisy signal $\mathbf{z}(t)$ is replaced by the clean signal $\mathbf{y}(t)$. Middle column: the distance matrix is defined in \eqref{eq:mahalanobis_distance estimate}, where the estimates of the precision matrix are based on the pseudo-inverse. Right: the third distance matrix is defined in \eqref{eq:mahalanobis_distance estimate}, where the estimates of the precision matrix are based on the proposed OS.
First row: the scatter plot of the slow variable $x_1$ and the most correlated eigenvector of the transition matrix associated with diffusion maps. Second row: the scatter plot of the fast variable $x_2$ and its most correlated eigenvector of the transition matrix associated with diffusion maps. 
The correlation between the respective variables and eigenvectors is shown in red above each subplot. 
Third row: the eigenvalues of diffusion maps, where $d_k$ is the $k$-th largest eigenvalue.
The eigenvalues corresponding to the eigenvectors presented in the first and second rows (corresponding to the slow and fast variables) are marked by red and blue circles, respectively.}
\label{plot_f1}
\end{figure}

\begin{figure}[hbt!]
\center
\includegraphics[trim=50 0 50 0, width=.7\textwidth]{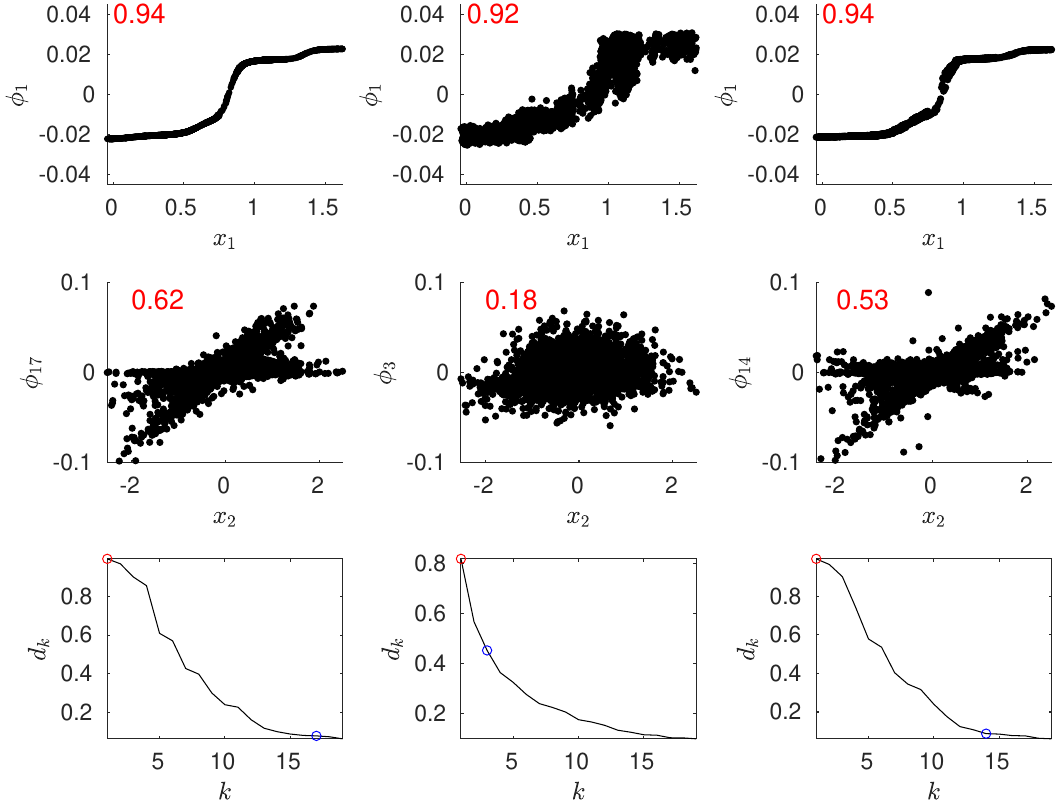}
\caption{The results obtained for Case II by diffusion maps with different distances. The meaning of the plot is the same as that detailed in the caption of Figure \ref{plot_f1}.}
\label{plot_f2}
\end{figure}

\begin{figure}[hbt!]
\center
\includegraphics[trim=50 0 50 0, width=.7\textwidth]{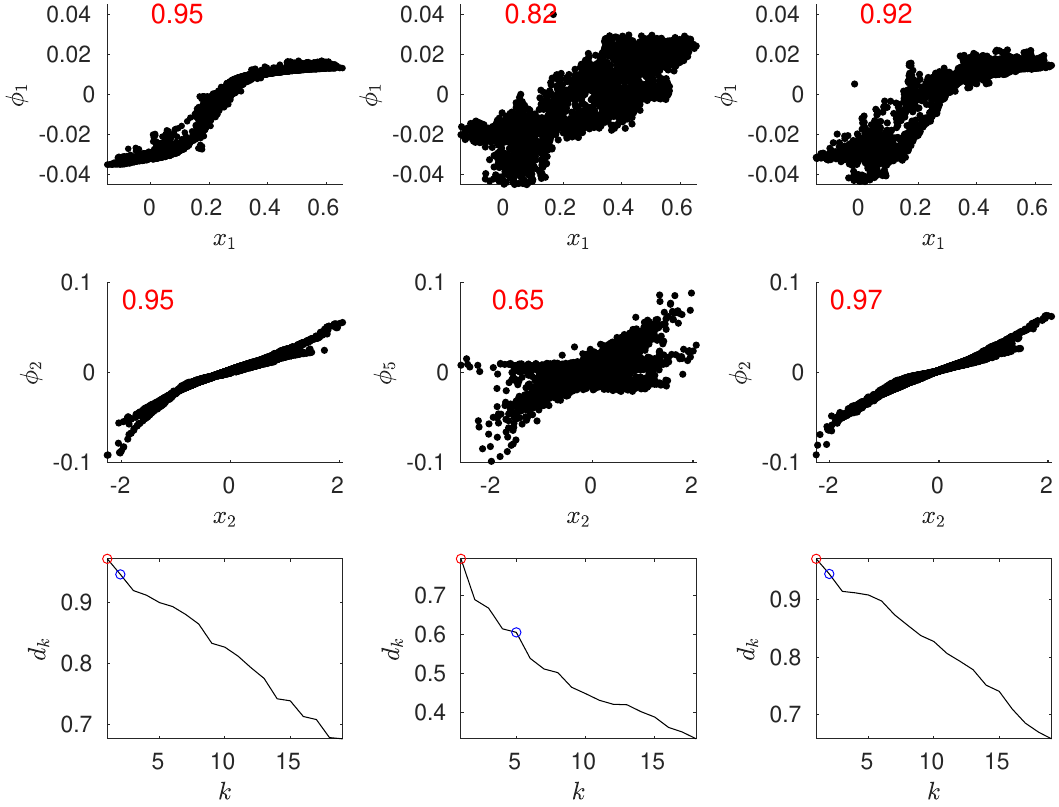}
\caption{The results obtained for Case III by diffusion maps with different distances. The meaning of the plot is the same as that detailed in the caption of Figure \ref{plot_f1}.}
\label{plot_f3}
\end{figure}

\begin{table*}[hbt!] 

\caption{The normalized loss of the optimal shrinker estimator $M_n = \tilde{\eta}^*(S_n)$ and the classical estimator $M_n = \eta^{\text{classical}}_{\sigma}(S_n)$ in the manifold setup. The mean and the standard deviation over $500$ realizations are reported.}
\centering
\begin{tabular}{cc||cc|cc}
\\
\multicolumn{2} {c||} {} & \multicolumn{2} {c|} {$\texttt{Error}(M_n,y_1)$} &  \multicolumn{2} {c} {$\texttt{Error}(M_n,y_2)$} \\
& & $\eta^{\text{classical}}_{\sigma}(S_n)$ & $\tilde{\eta}^*(S_n)$ & $\eta^{\text{classical}}_{\sigma}(S_n)$ & $\tilde{\eta}^*(S_n)$  \\
\hline\hline
 & $\sigma = 1$ & 18.78 $\pm$ 1.16 & 0.78 $\pm$ 0.54 & 55.98 $\pm$ 2.09 & 1.32 $\pm$ 0.93\\
$\beta = 0.1$ & $\sigma = 1.5$ & 23.72 $\pm$ 1.19 & 1.41 $\pm$ 0.87 & 59.42 $\pm$ 2.10 & 2.59 $\pm$ 1.74\\
& $\sigma = 2$ & 33.54 $\pm$ 1.53 & 2.18 $\pm$ 1.31 & 64.04 $\pm$ 1.69 & 5.19 $\pm$ 2.99\\
\hline
 & $\sigma = 1$ & 26.86 $\pm$ 2.70 & 2.41 $\pm$ 1.55 & 60.88 $\pm$ 4.22 & 4.06 $\pm$ 2.54\\
$\beta = 0.5$ & $\sigma = 1.5$ & 42.66 $\pm$ 3.43 & 4.78 $\pm$ 2.65 & 69.06 $\pm$ 3.43 & 10.65 $\pm$ 5.38\\
& $\sigma = 2$ & 58.59 $\pm$ 3.24 & 9.84 $\pm$ 4.63 & 77.24 $\pm$ 2.81 & 31.52 $\pm$ 17.56\\
\hline
 & $\sigma = 1$ & 34.70 $\pm$ 4.89 & 4.05 $\pm$ 2.35 & 64.11 $\pm$ 5.57 & 8.39 $\pm$ 3.91\\
$\beta = 1$ & $\sigma = 1.5$ & 54.72 $\pm$ 4.46 & 10.62 $\pm$ 4.69 & 75.54 $\pm$ 3.63 & 23.97 $\pm$ 12.49\\
& $\sigma = 2$ & 69.65 $\pm$ 3.97 & 21.35 $\pm$ 7.94 & 83.30 $\pm$ 2.78 & 62.99 $\pm$ 19.34\\
\hline
\end{tabular}
\label{table}
\end{table*}

\section{CONCLUSIONS} \label{sec:conclusion}

We proposed a new estimator for \MD~based on precision matrix shrinkage. For an appropriate loss function, we show that the proposed estimator is asymptotically optimal and outperforms the classical implementation of \MD~using the Moore-Penrose pseudo-inverse of the sample covariance. Importantly, the proposed estimator is particularly beneficial when the data is noisy and in high-dimension, a case in which the classical \MD~estimator might completely fail. Consequently, we believe that the new estimator may be useful in modern data analysis applications, involving for example, local principal component analysis, metric design, and manifold learning. 

In this work, we focused on the case in which the intrinsic dimensionality of the data (the rank of the covariance matrix) $d$ is unknown, and therefore, it was not explicitly used in the estimation.
Yet, in many scenarios, this dimension is known. In this case, it could be beneficial to consider a direct truncation and use only the top $d$ eigen-pairs for the estimation of the precision matrix. 
While the benefit from such a truncation has been shown empirically in several applications \cite{talmon2015manifold,wu2015assess,yair2017reconstruction,liu2018diffuse}, it still requires a systematic investigation. For example, identifying the rank of the signal, or estimating the dimension of a manifold, are by themselves highly challenging tasks \cite{Kritchman:2009:NDN:1653500.1653516,levina2005maximum}.
Note that in the particular manifold setup, knowing the manifold dimension is essentially different from the {\em rank-aware shrinker} discussed in \cite{2013arXiv1311.0851D}; as we showed here, under the manifold setup, the rank of the covariance matrix associated with points residing inside a small neighborhood of any point on the manifold could be much larger than $d$. {Since the focus of this paper is \MD~recovery, the associated loss function for the OS of the precision matrix is the operator norm. As is discussed in \cite{2013arXiv1311.0851D}, there are other loss functions that we can choose, like the Frobenius norm, the nuclear norm, etc. It would be interesting to explore if the OS for the precision matrix under those norms when they are needed. We shall also mention that the manifold model discussed in this paper is a simplified model for more complicated datasets with more nonlinear structure. While it is possible to apply the principle component analysis approach to denoise the data when the smooth compact manifold assumption holds thanks to the fixed $K$ mentioned in Section \ref{Sec:manifoldModelOptShr}, in general this approach might not be optimal, particularly when the geometric structure is more complicated. The current work explores this situation with a simplified manifold model and paves a way toward more complicated setups, and these cases will be explored in our future work.

}

\section*{Acknowledgements}
{MG has been supported by H-CSRC Security Research Center and Israeli Science Foundation grant no. 1523/16. MG and RT were supported by a grant from the Tel-Aviv University ICRC Research Center. We thank the anonymous reviewers for their constructive comments and critiques that highly improve the original manuscript.}

	\bibliographystyle{plain}
	\bibliography{refs}

\begin{thebibliography}{10}

\bibitem{alagapan2018diffusion}
Sankaraleengam Alagapan, Hae~Won Shin, Flavio Fr{\"o}hlich, and Hau-Tieng Wu.
\newblock Diffusion geometry approach to efficiently remove electrical
  stimulation artifacts in intracranial electroencephalography.
\newblock {\em Journal of neural engineering}, 2018.

\bibitem{2004math......3022B}
J.~{Baik}, G.~{Ben Arous}, and S.~{Peche}.
\newblock {Phase transition of the largest eigenvalue for non-null complex
  sample covariance matrices}.
\newblock {\em Ann. Probab.}, 33(5):1643--1697, 2005.

\bibitem{BAIK20061382}
J.~Baik and J.~W. Silverstein.
\newblock Eigenvalues of large sample covariance matrices of spiked population
  models.
\newblock {\em Journal of Multivariate Analysis}, 97(6):1382 -- 1408, 2006.

\bibitem{belkin2003laplacian}
M.~Belkin and P.~Niyogi.
\newblock Laplacian eigenmaps for dimensionality reduction and data
  representation.
\newblock {\em Neural computation}, 15(6):1373--1396, 2003.

\bibitem{coifman2006diffusion}
R.~R. Coifman and S.~Lafon.
\newblock Diffusion maps.
\newblock {\em Applied and computational harmonic analysis}, 21(1):5--30, 2006.

\bibitem{10.1007/978-3-319-73241-1_4}
D.~Dai and T.~Holgersson.
\newblock High-dimensional {CLT}s for individual mahalanobis distances.
\newblock In M{\"u}jgan Tez and Dietrich {von Rosen}, editors, {\em Trends and
  Perspectives in Linear Statistical Inference}, pages 57--68, 2018.

\bibitem{2013arXiv1311.0851D}
D.~L. {Donoho}, M.~{Gavish}, and I.~M. {Johnstone}.
\newblock {Optimal Shrinkage of Eigenvalues in the Spiked Covariance Model}.
\newblock {\em The Annals of Statistics}, 46(4):1742--1778, November 2018.

\bibitem{Donoho_Grimes:2003}
D.~L. Donoho and C.~Grimes.
\newblock {Hessian eigenmaps: Locally linear embedding techniques for
  high-dimensional data}.
\newblock {\em P. Natl. Acad. Sci. USA}, 100(10):5591--5596, 2003.

\bibitem{dov2016kernel}
David Dov, Ronen Talmon, Israel Cohen, David Dov, Ronen Talmon, and Israel
  Cohen.
\newblock Kernel method for voice activity detection in the presence of
  transients.
\newblock {\em IEEE/ACM Transactions on Audio, Speech and Language Processing
  (TASLP)}, 24(12):2313--2326, 2016.

\bibitem{dsilva2016data}
Carmeline~J Dsilva, Ronen Talmon, C~William Gear, Ronald~R Coifman, and
  Ioannis~G Kevrekidis.
\newblock Data-driven reduction for a class of multiscale fast-slow stochastic
  dynamical systems.
\newblock {\em SIAM Journal on Applied Dynamical Systems}, 15(3):1327--1351,
  2016.

\bibitem{dsilva2013nonlinear}
Carmeline~J Dsilva, Ronen Talmon, Neta Rabin, Ronald~R Coifman, and Ioannis~G
  Kevrekidis.
\newblock Nonlinear intrinsic variables and state reconstruction in multiscale
  simulations.
\newblock {\em The Journal of chemical physics}, 139(18):11B608\_1, 2013.

\bibitem{ElKaroui:2010a}
N.~{El Karoui}.
\newblock {On information plus noise kernel random matrices}.
\newblock {\em Ann. Stat.}, 38(5):3191--3216, 2010.

\bibitem{ElKaroui_Wu:2016b}
N.~{El Karoui} and H.-T. {Wu}.
\newblock {Connection graph Laplacian methods can be made robust to noise}.
\newblock {\em Ann. Stat.}, 44(1):346--372, 2016.

\bibitem{johnstone2001}
I.~M. Johnstone.
\newblock {On the distribution of the largest eigenvalue in principal
  components analysis}.
\newblock {\em Annals of Statistics}, 29(2):295--327, 2001.

\bibitem{Johnstone2006}
I.~M Johnstone.
\newblock {High Dimensional Statistical Inference and Random Matrices}.
\newblock In {\em Proc. Int. Congr. Math.}, pages 307--333, 2006.

\bibitem{Kritchman:2009:NDN:1653500.1653516}
S.~Kritchman and B.~Nadler.
\newblock Non-parametric detection of the number of signals: Hypothesis testing
  and random matrix theory.
\newblock {\em Trans. Sig. Proc.}, 57(10):3930--3941, October 2009.

\bibitem{lahav2018mahalanobis}
Almog Lahav, Ronen Talmon, and Yuval Kluger.
\newblock Mahalanobis distance informed by clustering.
\newblock {\em Information and Inference: A Journal of the IMA}, 8(2):377--406,
  2018.

\bibitem{lederman2015learning}
Roy~R. Lederman and Ronen Talmon.
\newblock Learning the geometry of common latent variables using
  alternating-diffusion.
\newblock {\em Applied and Computational Harmonic Analysis}, 44(3):509 -- 536,
  2018.

\bibitem{levina2005maximum}
E.~Levina and P.~J. Bickel.
\newblock Maximum likelihood estimation of intrinsic dimension.
\newblock In {\em Advances in neural information processing systems}, pages
  777--784, 2005.

\bibitem{lin2021wave}
Yu-Ting Lin, John Malik, and Hau-Tieng Wu.
\newblock Wave-shape oscillatory model for nonstationary periodic time series
  analysis.
\newblock {\em Foundations of Data Science}, 2021.

\bibitem{Linderbaum2015}
Ofir {Lindenbaum}, Arie {Yeredor}, Moshe {Salhov}, and Amir {Averbuch}.
\newblock {MultiView Diffusion Maps}.
\newblock {\em arXiv preprint arXiv 1508.05550}, 2015.

\bibitem{liu2018diffuse}
G.-R. Liu, Y.-L. Lo, J.~Malik, Y.-C. Sheu, and H.-T. Wu.
\newblock Diffuse to fuse eeg spectra--intrinsic geometry of sleep dynamics for
  classification.
\newblock {\em Biomedical Signal Processing and Control, accepted for
  publication}, 2019.

\bibitem{mahalanobis1936generalized}
P.~C. Mahalanobis.
\newblock On the generalized distance in statistics.
\newblock {\em Proceedings of the National Institute of Sciences (Calcutta)},
  2:49--55, 1936.

\bibitem{2018arXiv180402811M}
J.~{Malik}, C.~{Shen}, H.-T. {Wu}, and N.~{Wu}.
\newblock {Connecting Dots -- from Local Covariance to Empirical Intrinsic
  Geometry and Locally Linear Embedding}.
\newblock {\em Pure and Applied Analysis}, 1(4):515--542, 2019.

\bibitem{0025-5734-1-4-A01}
V~A Marcenko and L~A Pastur.
\newblock Distribution of eigenvalues for some sets of random matrices.
\newblock {\em Mathematics of the USSR-Sbornik}, 1(4):457, 1967.

\bibitem{McLachlan1999}
G.~J. McLachlan.
\newblock Mahalanobis distance.
\newblock {\em Resonance}, 4(6):20--26, 1999.

\bibitem{nash1956imbedding}
John Nash.
\newblock The imbedding problem for riemannian manifolds.
\newblock {\em Annals of mathematics}, pages 20--63, 1956.

\bibitem{Paul2007}
D.~Paul.
\newblock {Asymptotics of Sample Eigenstructure for a Large Dimensional Spiked
  Covariance Model}.
\newblock {\em Statistica Sinica}, 17:1617--1642, 2007.

\bibitem{Roweis_Saul:2000}
S.~T. Roweis and L.~K. Saul.
\newblock Nonlinear dimensionality reduction by locally linear embedding.
\newblock {\em Science}, 290(5500):2323--2326, 2000.

\bibitem{singer2008non}
A.~Singer and R.~R. Coifman.
\newblock Non-linear independent component analysis with diffusion maps.
\newblock {\em Applied and Computational Harmonic Analysis}, 25(2):226--239,
  2008.

\bibitem{singer2009NLM}
A.~{Singer}, Y.~{Shkolnisky}, and B.~{Nadler}.
\newblock Diffusion interpretation of nonlocal neighborhood filters for signal
  denoising.
\newblock {\em SIAM J. Imaging Sciences}, 2(1):118--139, 2009.

\bibitem{SingerWu2012}
A.~Singer and H.-T. Wu.
\newblock {Two-dimensional tomography from noisy projections taken at unknown
  random directions}.
\newblock {\em SIAM J. Imaging Sciences}, 6(1):136--175, 2012.

\bibitem{Singer_Wu:2012}
A.~Singer and H.-T. Wu.
\newblock Vector diffusion maps and the connection {Laplacian}.
\newblock {\em Comm. Pure Appl. Math.}, 65(8):1067--1144, 2012.

\bibitem{singer2018mathematics}
Amit Singer.
\newblock Mathematics for cryo-electron microscopy.
\newblock In {\em Proceedings of the International Congress of Mathematicians:
  Rio de Janeiro 2018}, pages 3995--4014. World Scientific, 2018.

\bibitem{Stein1986}
C.~M. Stein.
\newblock {Lectures on the theory of estimation of many parameters}.
\newblock {\em Journal of Soviet Mathematics}, 74(5), 1986.

\bibitem{su2019recovery}
Pei-Chun Su, Stephen Miller, Salim Idriss, Piers Barker, and Hau-Tieng Wu.
\newblock Recovery of the fetal electrocardiogram for morphological analysis
  from two trans-abdominal channels via optimal shrinkage.
\newblock {\em Physiological measurement}, 40(11):115005, 2019.

\bibitem{talmon2013empirical}
R.~Talmon and R.~R. Coifman.
\newblock Empirical intrinsic geometry for nonlinear modeling and time series
  filtering.
\newblock {\em Proceedings of the National Academy of Sciences},
  110(31):12535--12540, 2013.

\bibitem{talmon2015manifold}
R.~Talmon, S.~Mallat, H.~Zaveri, and R.~R. Coifman.
\newblock Manifold learning for latent variable inference in dynamical systems.
\newblock {\em IEEE Transactions on Signal Processing}, 63(15):3843--3856,
  2015.

\bibitem{talmon2015intrinsic}
Ronen Talmon and Ronald~R Coifman.
\newblock Intrinsic modeling of stochastic dynamical systems using empirical
  geometry.
\newblock {\em Applied and Computational Harmonic Analysis}, 39(1):138--160,
  2015.

\bibitem{talmon2011parametrization}
Ronen Talmon, Dan Kushnir, Ronald~R Coifman, Israel Cohen, and Sharon Gannot.
\newblock Parametrization of linear systems using diffusion kernels.
\newblock {\em IEEE Transactions on Signal Processing}, 60(3):1159--1173, 2011.

\bibitem{talmon2016latent}
Ronen Talmon and Hau-Tieng Wu.
\newblock Latent common manifold learning with alternating diffusion: analysis
  and applications.
\newblock {\em Applied and Computational Harmonic Analysis}, 47(3):848--892,
  2019.

\bibitem{tenenbaum2000global}
J.~B. Tenenbaum, V.~De~Silva, and J.~C. Langford.
\newblock A global geometric framework for nonlinear dimensionality reduction.
\newblock {\em science}, 290(5500):2319--2323, 2000.

\bibitem{wang2020novel}
Shen-Chih Wang, Hau-Tieng Wu, Po-Hsun Huang, Cheng-Hsi Chang, Chien-Kun Ting,
  and Yu-Ting Lin.
\newblock Novel imaging revealing inner dynamics for cardiovascular waveform
  analysis via unsupervised manifold learning.
\newblock {\em Anesthesia \& Analgesia}, 130(5):1244--1254, 2020.

\bibitem{Weinberger:2009:DML:1577069.1577078}
K.~Q. Weinberger and L.~K. Saul.
\newblock Distance metric learning for large margin nearest neighbor
  classification.
\newblock {\em J. Mach. Learn. Res.}, 10:207--244, 2009.

\bibitem{wu2015assess}
H.-T. Wu, R.~Talmon, and Y.-L. Lo.
\newblock Assess sleep stage by modern signal processing techniques.
\newblock {\em IEEE Trans. Biomed. Engineering}, 62(4):1159--1168, 2015.

\bibitem{Wu_Wu:2017}
H.-T. Wu and N~Wu.
\newblock {Think globally, fit locally under the Manifold Setup: Asymptotic
  Analysis of Locally Linear Embedding}.
\newblock {\em Annuls of Statistics}, 46(6B):3805--3837, 2018.

\bibitem{Xiang2008LearningAM}
S.~Xiang, F.~Nie, and C.~Zhang.
\newblock Learning a {Mahalanobis} distance metric for data clustering and
  classification.
\newblock {\em Pattern Recognition}, 41:3600--3612, 2008.

\bibitem{yair2017reconstruction}
O.~Yair, R.~Talmon, R.~R. Coifman, and I.~G. Kevrekidis.
\newblock Reconstruction of normal forms by learning informed observation
  geometries from data.
\newblock {\em Proceedings of the National Academy of Sciences},
  114(38):E7865--E7874, 2017.

\bibitem{Yang2006DistanceML}
Liu Yang and Rong Jin.
\newblock Distance metric learning: A comprehensive survey.
\newblock {\em Michigan State Universiy}, 2(2):4, 2006.

\end{thebibliography}

    \end{document}